\documentclass[a4paper]{article}

\usepackage{
amsmath,
amsthm,
amscd,
amssymb,
wasysym
}
\usepackage[safe]{tipa}
\usepackage{comment}
\usepackage{tikz}
\usetikzlibrary{positioning,arrows,calc}
\usepackage{xspace}

\usepackage{graphicx}

\usepackage{mathrsfs}

\usepackage{url}

\numberwithin{equation}{section}
\theoremstyle{plain}
\newtheorem{lemma}[equation]{Lemma}
\newtheorem{theorem}[equation]{Theorem}
\newtheorem{proposition}[equation]{Proposition}
\newtheorem{corollary}[equation]{Corollary}

\newtheorem{remark}[equation]{Remark}

\theoremstyle{definition}
\newtheorem{definition}[equation]{Definition}

\newtheorem{example}[equation]{Example}

\newcommand{\R}{\mathbb{R}}

\newcommand{\Z}{\mathbb{Z}}
\newcommand{\C}{\mathcal{C}}
\newcommand{\D}{\mathcal{D}}
\newcommand{\N}{\mathbb{N}}

\newcommand{\ann}{\mathcal{A}}

\newcommand{\ball}{\mathcal{B}}

\newcommand\xqed[1]{%
  \leavevmode\unskip\penalty9999 \hbox{}\nobreak\hfill
  \quad\hbox{#1}}
\newcommand\qee{\xqed{$\fullmoon$}}

\newcommand{\pow}{\mathcal{P}}
\newcommand{\follow}{\mathcal{F}}
\newcommand{\pat}{\mathscr{P}}
\newcommand{\qat}{\mathscr{Q}}
\newcommand{\lang}{\mathscr{L}}

\newcommand{\Cay}{\textrm{Cay}}

\newcommand{\colo}{\;:\;}

\title{Avoshifts}

\author{
Ville Salo \\
vosalo@utu.fi
}

\begin{document}
\maketitle

\begin{abstract}
An avoshift is a subshift where for each set $C$ from a suitable family of subsets of the shift group, the set of all possible valid extensions of a globally valid pattern on $C$ to the identity element is determined by a bounded subpattern. This property is shared (for various families of sets $C$) by for example cellwise quasigroup shifts, TEP subshifts, and subshifts of finite type with a safe symbol. In this paper we concentrate on avoshifts on polycyclic groups, when the sets $C$ are what we call ``inductive intervals''. We show that then avoshifts are a recursively enumerable subset of subshifts of finite type. Furthermore, we can effectively compute lower-dimensional projective subdynamics and certain factors (avofactors), and we can decide equality and inclusion for subshifts in this class. These results were previously known for group shifts, but our class also covers many non-algebraic examples as well as many SFTs without dense periodic points. The theory also yields new proofs of decidability of inclusion for SFTs on free groups, and SFTness of subshifts with the topological strong spatial mixing property.
\end{abstract}


\section{Introduction}

Let $\Sigma$ be a finite set called the \emph{alphabet} and $G$ a discrete group. A \emph{subshift of finite type} or \emph{SFT} $X \subset \Sigma^G$ is a set closed under the \emph{shift maps} $\sigma_g(x)_h = gx_h = x_{g^{-1}h}$, and which is defined by removing exactly the \emph{points} $x \in \Sigma^G$ whose orbit hits a particular clopen set $F$, or equivalently contain a forbidden pattern from a finite set. 

It is well-known that one-dimensional SFTs (the case $G = \Z$) behave in a much nicer and uniform way than multidimensional SFTs ($G = \Z^d, d \geq 2$). As two examples of nice dynamical properties that hold in one dimension, topologically transitive subshifts of finite type in one dimension have a unique measure of maximal entropy and have dense periodic points (these fail for general subshifts of finite type, but in completely manageable ways). Furthermore, SFTs (and even their subshift factors called \emph{sofic shifts}) can be algorithmically compared for equality and inclusion. See e.g.\ the standard reference \cite{LiMa21}.

In higher dimensions, even under strong irreducibility (a strong mixing/niceness assumption), subshifts may have multiple measures of maximal entropy \cite{BuSt94}, $\Z^3$-SFTs are not known to have dense periodic points (see e.g.\ \cite{Ho24}), and it is not even possible to algorithmically tell whether a given subshift of finite type is empty \cite{Be66}. For this reason, a recurring endeavor in the field of symbolic dynamics is to try to find classes of multidimensional subshifts which capture some of the interesting phenomena from the two-dimensional setting, but where at least some properties are decidable or otherwise understandable.

Group shifts -- groups where the elements are points of the subshift, and group operations are continuous and shift-commuting -- are the most successful such a class. They are of finite type when the shift group $G$ is polycyclic \cite{Sc95}. They have dense periodic points at least on abelian groups \cite{KiSc88}. They also of course have a ``most uniform'' measure, namely the Haar measure. Many of their basic properties (inclusions, and even some dynamical properties) are decidable \cite{BeKa24}. An extensive theory of group shifts (with emphasis on the abelian case) is developed in \cite{Sc95}.

In \cite{Sa22d}, the author introduced the so-called $k$-TEP subshifts. They share some properties of group shifts and one-dimensional SFTs: there is a natural measure (the measure that ``samples patterns uniformly on convex sets''). They have dense periodic points, and there are also some decidable properties, for example inclusion is easily seen to be decidable for this class.\footnote{The last results are not explicitly stated in \cite{Sa22d}, but are obvious corollaries of results proved therein; in any case, they follow from the results proved in the present paper.} 

Different mixing assumptions can have useful implications also in higher dimensions (even if usually not as nice as in one dimension). Strong irreducibility implies dense periodic points at least for two-dimensional SFTs \cite{Li03,Li04}. Two stronger mixing properties are the existence of a \emph{safe symbol}, and the more general \emph{topological strong spatial mixing}. Here, we say $0 \in \Sigma$ is a safe symbol for an SFT $X$ if turning a non-$0$ symbol to $0$ in a valid configuration can never introduce a forbidden pattern, for the latter property see Definition~\ref{def:TSSM}. All SFTs in these classes dense periodic points, and thus their languages and inclusion are decidable. 

In the present paper, we define a new class of subshifts called \emph{avoshifts}. This class is defined with respect to a family of subsets of the acting shift group $G$, and the definition is that if a partial configuration $x \in \Sigma^C$ (with $C$ in the specified class of shapes) extends to a complete configuration of the subshift, then the set of valid extensions to the identity element is determined by a finite subpattern of the configuration, which is bounded as a function of $x$, but not necessarily $C$.


We consider the class of avoshifts for a quite restricted\footnote{This is a good thing -- the smaller the class of shapes, the larger the class of avoshifts.} class of shapes called inductive intervals. With this choice, avoshifts generalize (cellwise quasi-)group shifts, $k$-TEP subshifts (on $\Z^d$ with the standard convex geometry). It also covers all SFTs with a safe symbol, and topological strong spatial mixing in fact precisely coincides with the avo property for the class of all possible shapes. Inductive intervals make sense on any polycyclic group, and depend on the chosen subnormal series.


The following is our main ``practical'' result.

\begin{theorem}
\label{thm:Main}
Let $G$ be polycyclic, and let $X \subset A^G$ be an avoshift for inductive intervals. Then
\begin{itemize}
\item $X$ is of finite type,
\item the language of $X$ can be computed algorithmically (uniformly in $X$),
\item $X \subset Y$ can be decided for any given SFT $Y$,
\item forbidden patterns for the restrictions of $X$ to the groups in the subnormal series (a.k.a.\ projective subdynamics) can be effectively computed,
\item ``avofactors'' can be effectively computed,
\item $X$ is uniformly SFT on inductive intervals, and the corresponding forbidden patterns can be computed, and
\item we can algorithmically semidecide that $X$ is an avoshift from its forbidden patterns.
\end{itemize}
\end{theorem}

Avofactors are factors that can be expressed as projections from a product that satisfies a similar property as that defining avoshifts. This includes all factors given by algebraic maps between group shifts.

The first five results (see below for a discussion of the sixth) were previously known for group shifts (at least for abelian groups). The first is proved in \cite{Ki87,Sc95} (in the former on $\Z$, in the latter for general polycyclic groups), the second and third are from \cite{KiSc88}, and the fourth and fifth are from \cite{BeKa24}.

An important takeaway are that avoshifts on polycyclic groups satisfy the basic decidability properties, and SFTness, of group shifts. Notably, our proof of these results uses the group structure ``only once'', and the only fact used is ``equal extension counts'' (see Section~\ref{sec:EEC} for the definition) for all finite patterns of the same shape, which is trivially true for (cellwise quasi)group shifts. For group shifts, both \cite{KiSc88,BeKa24} strongly use Wang's algorithm \cite{Wa61}, which is based on the density of periodic points. Avoshifts do not have dense periodic points, and our methods are not related to Wang's algorithm in any obvious way.

The fact that avoshifts are of finite type does not seem to generalize beyond polycyclic groups \cite{Sa18e}, but the other properties are really consequences of the uniform avo property (meaning the set of valid extensions of $x \in\Sigma^C$ to the identity element is determined by a subpattern which is bounded as a function of $C$ and $x$), which in the case of polycyclic groups and inductive intervals is automatic. We illustrate this by giving an avoshift proof of the well-known fact that SFTs on free groups have decidable languages and decidable inclusion, by showing that SFTs are characterized by the uniform avo property on the tree convex sets from \cite{Sa22d}. 

The sixth property in Theorem~\ref{thm:Main} means that there is a finite set of forbidden patterns such that for $C \subset G$ in the family, configurations that are locally legal on $C$ are globally legal (similarly to the main property of $k$-TEP subshifts described in the first bullet point of \cite[Theorem~1.1]{Sa22d}). This is closely tied with the uniform avo property.

We also make the following two observations that may be of interest despite not having decidability implications: The topological Markov property \cite{BaGoMaTa20} is equivalent to being avo for the class of cofinite sets (Proposition~\ref{prop:TMP}). The property of being topologically $k$-mixing for all $k$ can be stated equivalently as being uniformly SFT on the sets of cardinality $n$ for all $n$.

A lengthy discussion of practicality of these methods, possible extensions of the theory, and future work is included in Section~\ref{sec:Discussion}.

\section{Definitions}

We have $0 \in \N$. We denote the power set of a set $S$ by $\pow(S)$. A subset (not necessarily proper) is written $A \subset B$. A finite subset is written $A \Subset B \iff A \subset B \wedge |A| < \infty$. We write $\Z_k$ for the cyclic group on $k$ elements, for $k$ finite. By $\vec 0$ we denote the zero vector of $\Z^d$ for any $d$. 

Throughout, $G$ will denote a group. Our groups are for the most part finitely generated and countable. The identity element is denoted by $e_G$ or just $e$. We typically think of $G$ as acting on itself by left translation. We consider the group to carry a generating set (usually called $S$ when a name is needed, but this symbol is not reserved for this purpose) and the resulting left-invariant word metric (typically called $d$). The \emph{ball of radius $n$} or \emph{$n$-ball} is $\ball_n = \{g \in G \colo d(e_G, g) \leq n\}$. Specifically this is the metric closed ball of radius $n$ around the identity element, but $g\ball_n$ is the metric ball around the element $g$ so we do not need special notation for it. We usually have in mind the right Cayley graph $\Cay(G, S)$ with nodes $G$ and edges $(g, gs)$ for $s \in S$. Subsets of $G$ are sometimes called \emph{areas} or \emph{shapes}. For $A, B \subset G$ we write $d(A, B) = \min\{d(a, b) \colo a \in A, b \in B\}$.

An \emph{alphabet} is a finite (discrete) set $\Sigma$. A \emph{subshift} is $X \subset \Sigma^G$ which is topologically closed ($\Sigma^G$ having the product topology) and closed under the \emph{shift} or \emph{translation} action of $G$ defined by formula $gx_h = x_{g^{-1}h}$, where in turn $x_h$ is the notation for indexing $x$ at $h$. Note that $\Sigma^G$ is compact, and thus so are subshifts.

The subsets $\pow(G)$ also carry the standard Fell topology (through identification of a set $S \subset G$ with its indicator function in $\{0,1\}^G$, this is the same as the topology used for subshifts). Subsets of $G$ are also considered as an ordered set with inclusion as the order, and $S_i \nearrow S$ means $\forall i: S_i \subset S_{i+1} \wedge \bigcup_i S_i = S$.

In this paper, we use the terminology that a \emph{point} is always an element of $\Sigma^G$ for a group $G$, a \emph{configuration} is any element of $\Sigma^D$ where $D \subset G$ for some group $G$ (possibly a point, possibly $D \subsetneq G$), and a \emph{pattern} is a configuration $x \in \Sigma^D$ whose domain $D$ is finite. We sometimes use the term ``finite pattern'' (meaning pattern), and ``partial configuration'' (meaning configuration) for emphasis, and a point is sometimes referred to as a ``complete configuration''. A \emph{$D$-pattern} is pattern with domain $D$. We mostly use $x, y, z$ to refer to points and configurations, and $P, Q, R$ to refer to patterns.

If $x \in \Sigma^D$ and $C \subset D$, we write restriction as $x|C \in \Sigma^C$, as we tend to have non-trivial formulas for the sets $C$ that we prefer not to have in a subscript. If $X \subset \Sigma^D$, $X|C = \{x|C \colo x \in X\}$. We say a configuration $x \in \Sigma^C$ is a \emph{subconfiguration} of $y \in \Sigma^D$ if $C \subset D$ and $y|C = x$. If $C$ is the set of elements of $D$ at some distance from the origin, we also call $x$ a \emph{prefix} of $y$. We say two configurations $x \in \Sigma^C, y \in \Sigma^D$ \emph{agree} on $B \subset C \cap D$ if $x|B = y|B$.

For a symbol $a \in \Sigma$ and $g \in G$, write $a^g$ for the pattern $P \in \Sigma^{\{g\}}$ with $P_g = a$. We sometimes identify $a$ with the pattern $a^e$. For two patterns $P \in \Sigma^C, Q \in \Sigma^D$ write $P \sqcup Q$ for the pattern $R \in \Sigma^{C \cup D}$ defined by $R|C = P, R|D = Q$, when such a pattern exists. If $P \in \Sigma^D$ and $g \in G$, then $gP \in \Sigma^{gD}$ is the configuration defined by $gP_{gh} = P_h$ (equivalently $gP_h = P_{g^{-1}h}$, so shifting a point $x \in \Sigma^G$ is a special case of this definition). The \emph{empty pattern} is the unique pattern with domain $\emptyset$.

We say $x$ \emph{occurs} or \emph{appears} in $y$ if $gx$ is a subconfiguration of $y$ for some $g \in G$; equivalently we say $y$ \emph{contains} $x$. We denote this by $x \sqsubset y$. We write $x \sqsubset X$ (and use the previous three terms as well) if $x \sqsubset y$ for some $y \in X$. We write $\lang(X) = \{P \colo P \sqsubset X\}$ for the \emph{language} of $X$. Note that the empty pattern is in the language of $X$ if and only $X$ is nonempty.

A \emph{subshift of finite type} or \emph{SFT} is $X \subset \Sigma^G$ defined by $X = \{x \in \Sigma^G \colo \forall g \in G: gx \notin C\}$, where $C \subset \Sigma^G$ is clopen. A clopen set is finite union of cylinders $[P] = \{x \in A^G \colo x|D = P\}$ where $P : D \to A$ for $D \Subset G$ (the notation means finite subset). The patterns $P$ giving the cylinders comprising $F$ are called the \emph{forbidden patterns}. A \emph{window} for an SFT $X \subset \Sigma^G$ is any $M \Subset G$ such that there is a set of defining forbidden patterns whose domains are contained in $M$. A \emph{window size} is any $m$ such that $\ball_m$ is a window. Note that on a group, we can always pick $M$ to be the same for all the finitely many forbidden patterns, but it will be necessary later (when we later talk about being SFT on a subset of the group) to allow for the forbidden patterns to have different domains.

If $M \Subset G$ and $X \subset \Sigma^G$ is a subshift, the \emph{$M$-SFT approximation} of $X$ is the SFT $Y \subset \Sigma^G$ defined by
\[ Y = \{y \in \Sigma^G \colo \forall g \in G: \exists x \in X: gy|M = x|M\}. \]
Equivalently, this is the set of configurations defined by forbidding precisely the patterns $P \in \Sigma^M$ which do not appear in $X$. 

Suppose a group $G$ and a subshift $X \subset \Sigma^G$ are clear from context. Then we say a configuration $x \in \Sigma^C$ is \emph{globally valid} or \emph{globally legal} if $x \in X|C$. In the case where $X$ is an SFT, and $\pat$ is a defining set of forbidden patterns, then $x \in \Sigma^C$ is \emph{locally $\pat$-valid} or \emph{locally $\pat$-legal} if no $P \in \pat$ appears in $x$ (with $\pat$ omitted if clear from context). Globally valid configurations are of course locally valid, and locally valid points (complete configurations) are globally valid.

If $X \subset \Sigma^G$ is a subshift and $x \in \Sigma^D$ for $D \subset G$, then the \emph{follower set} $\follow(x, E) = \follow_X(x, E)$ is the set of configurations $z \in \Sigma^E$ such that for some $y \in X$ we have $y|D = x \wedge y|E = z$, or equivalently $x \sqcup z \in X|D \cup E$ (and $\sqcup$ is well-defined). We refer to configurations $z \sqcup x \in X|D \cup E$ as the \emph{extensions} of $x$ to $E$. If $E$ is a singleton (indeed even if $E = \{g\} \neq \{e\}$), we may simply write the unique element of $\Sigma$ in this place to refer to an extension. 

We say a subshift $X \subset \Sigma^G$ is \emph{topologically mixing} if for any nonempty open $U, V \subset X$, for any large enough $g$ (in word norm) there is a point $x \in X$ with $x \in U, gx \in V$. More generally \emph{topological $k$-mixing} means for any nonempty open $U_1, \ldots, U_k$ there exists $n$ such that for any $(g_1, g_2, \ldots, g_k) \in G^k$ with $d(g_i, g_j) > n$ for $i \neq j$, there exists $x \in X$ such that $g_ix \in U_i$ for all $i$. In the introduction we used the term \emph{mixing/gluing property}; by this we refer to notions that in some (intuitive, informal) sense generalize the idea of topological mixing.

If $X, Y \subset \Sigma^G$ are subshifts, a \emph{conjugacy} is a shift-commuting homeomorphism $\phi : X \to Y$, and we say $X$ and $Y$ are \emph{conjugate} in this case. A \emph{factor map} is a shift-commuting continuous surjection $\phi : X \to Y$ and we say $X$ \emph{covers} $Y$ or $Y$ is a \emph{factor} of $X$ in this case. A factor of an SFT is called \emph{sofic} (note that in this paper, a factor is always a subshift).

If $\Sigma$ itself is a finite group, then a subshift $X \subset \Sigma^G$ is called a \emph{group shift} if it is closed under the operations obtained from those of $\Sigma$ by applying them cellwise. Group shifts are up to conjugacy the same as subshifts on which one can define a group structure by shift-invariant continuous operations \cite{Ki87,SaTo12e,SaTo22}. More generally, for any universal algebraic structure on $\Sigma$, we can talk about subshifts with this structure by applying the operations cellwise.

\section{Avoshifts and basic properties}

We begin with the precise definition of an avoshift. This is given for general groups and families of 
subsets of groups. 


\begin{definition}
Let $G$ be a group. Let $X \subset \Sigma^G$ be a subshift. We say $X$ is \emph{avo} for $C \subset G \setminus \{e_G\}$ 
(or \emph{$C$-avo}, or a \emph{$C$-avoshift}) if 
there exists $B\Subset C$ such that for all $x \in X|C$ we have
\[ x|B \sqcup a^e \sqsubset X \iff x \sqcup a^e \sqsubset X. \]
In this case we say $B$ is \emph{determining for $C$}, or \emph{$C$-determining}. Similarly we say $C$ is \emph{determined by $B$}, or \emph{$B$-determined}.
Let $\C \subset \pow(G \setminus \{e_G\})$. We say $X$ is \emph{$\C$-avo}, or a \emph{$\C$-avoshift} if it is avo for all $C \in \C$. 
\end{definition}



Note that to say $B$ is $C$-determining is a bit of an abuse of language: $C$-determining means that $B$ determines the possible continuations of (patterns with domain) $C$ to the identity, not that $B$ determines $C$ in any way. We use this terminology as it is concise, and hopefully not too confusing.


\begin{example}
\label{ex:Convex}
In the case of $G = \Z^d$, one natural family $\C$ we could use comes from Euclidean convexity. If $\D$ is the set of intersections $C \cap \Z^d$ where $C \subset \R^d$ is convex (for simplicity we refer to such subsets of $\Z^d$ as convex as well), then we take $\C$ to be the set of all $C \in \D$ such that also $C \setminus \{\vec 0\} \in \D$. For this choice, being an avoshift on $\Z^d$ equivalently means that if we take a convex subset $C \not\ni \vec 0$ of $\Z^d$, and add a new element $\vec v$ so that $C \cup \{\vec v\}$ stays convex, then for globally valid patterns $x$ on $C$, the set of valid extensions to $\vec v$ is determined by looking at only a bounded subpattern of $x$ (bounded as a function of $x$). This is the natural family to use for the $k$-TEP subshifts from \cite{Sa22d}, discussed in Section~\ref{sec:Examples}. \qee
\end{example}

\begin{example}
The class we concentrate on in the present paper is ``inductive intervals''. In $\Z^d$, an \emph{inductive interval} is defined by including all elements $\vec v \in \Z^d$ where the last coordinate is in a particular interval of one of the forms
\[ (-\infty, -1], [-m, -1], \emptyset, [1, m], [1, \infty) \]
or if it is zero, then the projection to the first $d-1$ coordinates belongs to an inductive interval of $\Z^{d-1}$. These are a proper subset of the convex sets defined in the previous example, so the corresponding notion of avoshifts covers a larger class of subshifts. A typical inductive interval for $\Z^3$ is illustrated in Figure~\ref{fig:minecraft}. \qee
\end{example}

We call these inductive intervals because they are defined inductively, and we prove most of their properties by induction.

\begin{figure}
\centering
\includegraphics[scale=0.25]{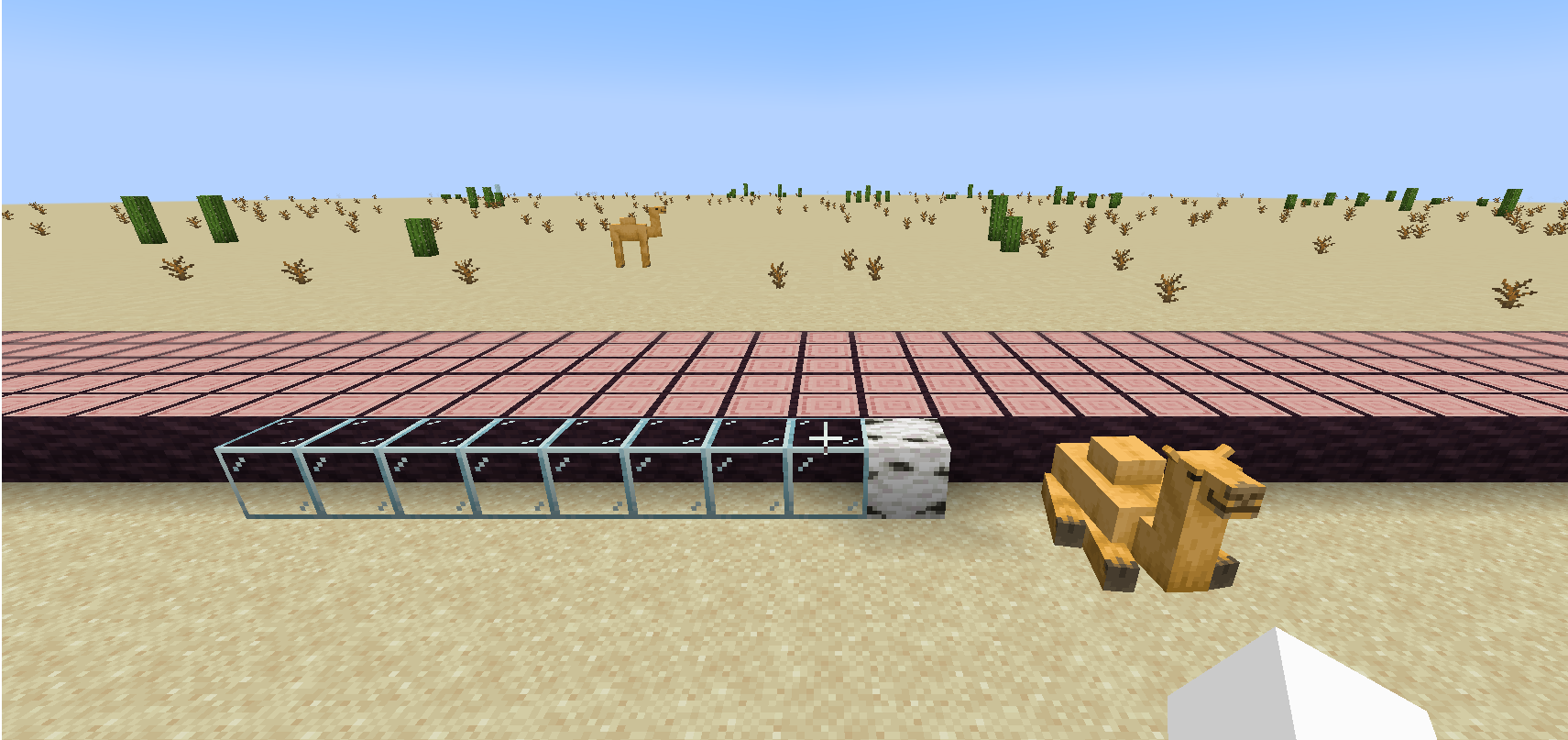}
\caption{The group $\Z^3$ visualized in Minecraft 1.21.1 \cite{minecraft}, with the first axis pointing right, the second axis forward, and the third axis upward. The inductive interval with axis intervals $([-8, -1], [1, 5], [-64, -1])$ is filled with blocks: a block of birch marks the origin, glass is used to fill on the first axis, cherry tree trunks on the second, and desert sand on the last. Plants and camels appear organically, and serve no mathematical purpose.}
\label{fig:minecraft}
\end{figure}

Another convenient way of expressing that $X$ is $C$-avo is through follower sets.

\begin{lemma}
A subshift $X \subset A^G$ is $C$-avo for $C \subset G \setminus \{e_G\}$ with determining set $B$ if and only if $\follow(x|C, e_G) = \follow(x|B, e_G)$ for all $x \in X$.
\end{lemma}

\begin{proof}
Suppose first $X$ is $C$-avo, and $B$ is $C$-determining. If $a \in \follow(x|C, e_G)$ then there exists $y \in X$ such that $y|C = x|C$ and $y_e = a$. Then also $y|B = x|B$ and $y_e = a$, so $a \in \follow(x|B, e_G)$. Conversely, if $a \in \follow(x|B, e_G)$ then there exists $y \in X$ such that $y|B = x|B$ and $y_e = a$. Thus, $x|B \sqcup a^e \sqsubset X$. Of course $x|C \in X|C$ and $x|C|B = x|B$. By the $C$-avo property,
\[ x|B \sqcup a^e \sqsubset X \iff x|C \sqcup a^e \sqsubset X. \]
so $x|C \sqcup a^e \sqsubset X$, meaning $a \in \follow(x|C, e_G)$.

Suppose then $\follow(x|C, e_G) = \follow(x|B, e_G)$. We show $X$ is $C$-avo with determining set $B$. Suppose $x \in X|C$, let $x = y|C$ for $y \in X$. We need to show
\[ x|B \sqcup a^e \sqsubset X \iff x \sqcup a^e \sqsubset X. \]
Suppose first $y|B \sqcup a^e = x|B \sqcup a^e \sqsubset X$. Then $\follow(y|B, e) \ni a$, so by the assumption also $a \in \follow(y|C, e) = \follow(x, e)$, meaning there exists $z \in X$ with $z|C = x, z_e = a$. Conversely, if $x \sqcup a^e \sqsubset X$, then trivially $x|B \sqcup a^e \sqsubset X$.
\end{proof}

Note that $e_G$ does not really play a special role in the definition of the avo property, since $X$ is always a subshift, and thus if $B \Subset C$ satisfies $\follow(x|C, e_G) = \follow(x|B, e_G)$ for all $x \in X$, then equivalently $\follow(x|gC, g) = \follow(x|gB, g)$ for all $g \in G$. Sometimes it simplifies life not to have to translate things back to the origin, and we use the preposition ``at'' as follows:

\begin{definition}
If $\follow(x|C, g) = \follow(x|B, g)$, then we say \emph{$B$ is determining for $C$ (or $C$-determining) at $g$}.
\end{definition}

\begin{lemma}
The follower set function $\follow$ is decreasing in the leftmost argument with respect to the subconfiguration relation.
\end{lemma}

\begin{proof}
Let $C \subset D$. Let $x \in X|C$ and $y \in X|D$ with $y|C = x$. Suppose $z \in \follow(y, B)$ for some $B \subset G$. This means there exists $w \in X$ such that $w|D = y$ and $w|B = z$. Then $w|C = y|C = x$, so $z \in \follow(x, B)$, concluding the proof.
\end{proof}

By the previous lemma, $\follow(x|C, e_G) = \follow(x|B, e_G)$ for $B \Subset C$ implies that $\follow(x|C, e_G) = \follow(x|D, e_G)$ for all $B \subset D \subset C$. We obtain the following immediate consequence:

\begin{lemma}
If $B \Subset C$ is $C$-determining, then $B$ is $D$-determining for all $D$ such that $B \subset D \subset C$, and every $D$ with $B \subset D \Subset C$, is $C$-determining.
\end{lemma}

We call such $D$ \emph{intermediate sets}. It is useful to observe that if $X$ is $C$-avo, then, since the determining set $B$ can be any sufficiently large finite intermediate set, we may simply take $B$ to be the set of elements in $C$ which are at distance $n$ from the identity in the group, for any sufficient $n$. Any $n$ such that $C \cap \ball_n$ is $C$-determining is called an \emph{avoradius} for $C$ (the term ``radius'' comes from the theory of cellular automata, see Proposition~\ref{prop:Spacetime}). This allows us to express the following notion, which plays a key role later in the paper. 

\begin{definition}
Let $\C \subset \pow(G \setminus \{e_G\})$ be any family of sets. We say $X$ is \emph{uniformly $\C$-avo} or a \emph{uniform $\C$-avoshift} if there is a common avoradius for all $C \in \C$.  
\end{definition}

In other words, if we have a configuration with domain $C \in \C$ that is globally valid, then to know the set of its legal continuations to $e_G$, it suffices to look at the subpattern on $C \cap \ball_n$, where $n$ does not depend on $C$.

The following characterization of the avo property is mainly used to explain the name, and to point out a connection to a well-known property of SFTs in one dimension, namely that the shift map being open characterizes one-sided SFTs \cite[Theorem~1]{Pa66a}. For $C \subset G \setminus \{e_G\}$, write $\dot C = C \cup \{e_G\}$.

\begin{lemma}
A subshift $X$ is avo for a set $C \subset G \setminus \{e_G\}$ if and only if the projection from $X|\dot C$ to $X|C$ is an open map.
\end{lemma}

Here, the set $X|D \subset \Sigma^D$ is topologized by the relative topology inherited from the product topology on $\Sigma^D$.

\begin{proof}
Suppose $B \Subset C$ is such that for all $x \in X|C$ we have
\[ x|B \sqcup a^e \sqsubset X \iff x \sqcup a^e \sqsubset X. \]
Note that, as explained above, this property holds for any superset of $B$ (contained in $C$) as well.

Write $\pi : X|\dot C \to X|C$ for the projection. Observe that this map is surjective, since $X|\dot C$ and $X|C$ are both just projections of $X \subset A^G$ to different sets of coordinates. Let $U \subset X|\dot C$ be an open set. Let $z \in \pi(U)$, say $z = \pi(y)$ where $y \in U$. It suffices to show that some open set (in $X|C$) containing $z$ is contained entirely in $\pi(U)$.

Since $y \in U$ and $U$ is open in $X|\dot C$, there exists $B' \Subset \dot C$ such that whenever $y'|B' = y|B'$ and $y' \in X|\dot C$, then $y' \in U$. Note that this property holds for any superset of $B'$ (contained in $\dot C$) as well. Thus, we may assume $\dot B = B \cup \{e_G\} = B'$ by possibly replacing both with a larger set. We claim that then the open set $[z|B]$ (in $X|C$) is contained in $\pi(U)$. To see this, let $x \in [z|B]$ be arbitrary, meaning $x \in X|C$ and $x|B = z|B$.

We now have $z \sqcup a^e = y \in X|\dot C$ where $a = y_e$. Thus trivially, $x|B \sqcup a^e = z|B \sqcup a^e \sqsubset X$. Thus, by the defining property of $B$, $x \sqcup a^e \in X|\dot C$. 
This means there exists $w \in X$ such that $w|\dot C = x \cup a^e$. Then $\pi(w|\dot C) = x$ and $w|\dot B = y|\dot B$, so $w|\dot C \in U$, concluding the proof.

Conversely, suppose the property fails for all $B \Subset C$, meaning for some $a \in \Sigma$ we can find, for all $B \Subset C$, patterns $x^B, y^B \in X|C$ such that $x^B|B = y^B|B$ and $x^B \sqcup a^e \in X|\dot C$ but $y^B \sqcup a^e \notin X|\dot C$. Let $(x, y)$ be any limit point of $(x^B, y^B)$. Then clearly $x = y \in X|C$ and $z = x \sqcup a^e = \lim_B (x^B \sqcup a^e) \in X|\dot C$. For the open set $[z|\{e\}]_{\{e\}} \ni z$, $y = x \in U|C$, but none of the points $y^B$ are in $U|C$, so $U|C$ it is not a neighborhood $y$. Thus, the projection is not open.
\end{proof}

\begin{remark}
The prefix ``avo'' \textup{[\textipa{Avo}]} refers to ``open'' in Finnish. This and the previous lemma are the source of the term ``avoshift''.
\end{remark}





The forward direction of the following result (that avoshifts are of finite type) is based on the classical proof of Kitchens for group shifts being subshifts of finite type \cite{Ki87}.

\begin{lemma}
\label{lem:ZCase}
If $G = \Z$, $X \subset \Sigma^G$ is an avoshift for the sets $(-\infty, -1]$ if and only if it is of finite type.
\end{lemma}

\begin{proof}
Suppose first that $X$ is of finite type. Let $n$ be the maximal diameter of a forbidden pattern. Then $(-\infty, -1]$ is determined by $[-n, -1]$: if $x \in X|(-\infty, -1]$ and $x|[-n, -1] \sqcup a^0 \sqsubset X$ then let $y \in X$ be any point with $y|[-n, 0] = x|[-n, -1] \sqcup a^0$. The point $z = x \sqcup y \in \Sigma^\Z$, i.e.\ the point defined by $z|(-\infty, -1] = x, z|[-n, \infty) = y|[-n, \infty)$ is clearly locally valid, and thus a point of $X$, and $z|(-\infty, 0] = x \sqcup a^0$, proving the avo property for $(-\infty, -1]$.

Suppose then that $X$ is an avoshift for this set. By assumption, then there is a determining set $M$ for $(\infty, -1]$, we can take $M = [-m, -1]$ and $\dot M = [-m, 0]$. 
We claim that then $X$ is equal to its $[-m, 0]$-SFT approximation. Equivalently, we claim that if every $i+\dot M$ subpattern of a point $x \in \Sigma^\Z$ is a translate of a pattern in $X|\dot M$, then $x \in X$. For this, suppose that indeed $x \in \Sigma^\Z$ is such that $x|i+\dot M \in X|i+\dot M$ for all $i \in \Z$.

For $n \in \Z$, consider the intervals $J_0 = [n, n+m-1], J_1 = [n, n+m], J_2 = [n, n+m+1], \ldots$. 
We show by induction that $x|J_i \sqsubset X$ for all $i$. For the basis of induction, our assumption directly implies $x|J_0 = x|n+m+M \sqsubset X$. Suppose then that $x|J_i$ is globally valid, and let $y \in X$ satisfy $y|J_i = x|J_i$. We have in particular $y|n+m+i+M \sqsubset X$ since $n+m+i+M \subset J_i$. Since $M$ is determining at $0$ for $(-\infty, -1]$, the set $n+m+i+M$ is determining at $n+m+i$ for $(-\infty, n+m+i-1]$. Thus it is also determining at $n+m+i$ for the intermediate set $J_i$ (recall that this means $n+m+i+M \subset J_i \subset (-\infty, n+m+i-1]$). Since $y|(-\infty, n+m+i-1] \sqsubset X$ (because even $y \in X$) and $x|n+m+i+\dot M \sqsubset X$, we conclude that $y|(-\infty, n+m+i-1] \sqcup (x_{n+m+i})^{n+m+i} \sqsubset X$, in particular $x|J_{i+1}$ is globally valid.

By compactness, we conclude that $x|\bigcup_i J_i = x|[n, n+\infty)$ is globally valid. Since $n \in \Z$ was arbitrary, we conclude that $x \in X$.
\end{proof}



The direction that all SFTs are avo is specific to dimension $1$ (or rather, it generalizes naturally to free groups, see Theorem~\ref{thm:FreeGroupSFTUniformlyAvo}). As avoshifts have nice computational properties, no decidability-theoretically useful notion of avoshift can cover all SFTs in higher dimensions. The direction that avoshifts are SFT on the other hand generalizes to all polycyclic groups, as we will see.

Note that the previous lemma implies that even in the case of $\Z$, avoshifts need not have dense periodic points, and may have any number of measures of maximal entropy (though not in the topologically mixing case).

\subsection{Equal extension counts}
\label{sec:EEC}

In this section, we define subshifts with equal extension counts. This is an intermediate class that contains all group shifts and $k$-TEP subshifts (as we show in Section~\ref{sec:Examples}), and is contained in the avoshifts. In the present paper, all results we prove are proved for general avoshifts (or uniform avoshifts), so equal extension counts do not play a direct role in the proofs.

\begin{definition}
A subshift $X$ has \emph{equal extension counts} for a set $C \subset G \setminus \{e_G\}$ if the cardinality $|\follow(x|C, e_G)|$ does not depend on $x \in X$. If $\C$ is a family of subsets of $G$, we say $X$ has equal extension counts for $\C$ if it has equal extension counts for all $C \in \C$.
\end{definition}

The following proof again should remind the reader of Kitchens' argument that group shifts are of finite type.

\begin{lemma}
If $X$ has equal extension counts for a set $C \subset G \setminus \{e_G\}$, then it is avo for this set.
\end{lemma}

\begin{proof}
A map between compact spaces with fibers of constant finite cardinality is open. Or more concretely, $|\follow(x|S', e_G)|$ can only decrease as $S' \nearrow S$, so this is reached by some finite $S'$, and by uniformity $S'$ must be $S$-determining. 
\end{proof}

Again, we may define a notion of \emph{uniform equal extension counts} by requiring that the $B \Subset C$ in the definition such that $|\follow(x|C, e_G)| = |\follow(x|B, e_G)|$ can be taken $C \cap \ball_r$ for some fixed $r$. Uniform equal extension counts for a family of sets $\C$ implies uniform avo for the same set, by the previous proof.

\section{Examples of avoshifts}
\label{sec:Examples}

Our first example are the group shifts. The proof should again should remind the reader of Kitchens' argument \cite{Ki87} that group shifts are of finite type.

\begin{lemma}
\label{lem:Quasigroup}
Let $\Sigma$ be a finite quasigroup, and suppose $X \subset \Sigma^G$ is a subshift closed under cellwise application of operations of $\Sigma$. Then $X$ has equal extension counts for any $C \subset G \setminus \{e_G\}$. Thus, such $X$ is an avoshift for $\pow(G \setminus \{e_G\})$.
\end{lemma}

\begin{proof}
We extend the cellwise quasigroup operations to any patterns with the same domain in the usual way (the operations are defined cellwise, so apply them cellwise in the input patterns). Let $C \subset G \setminus \{e_G\}$ be arbitrary. If $x, y \in X|C$ then $z \cdot x = y$ where $z = (y / x)$. List the extensions $x_1, \ldots, x_k \in \follow(x|C, C\cup\{e_G\})$. Let $z' \in \follow(z, C \cup \{e_G\})$ be arbitrary. Then clearly $(z' \cdot x_i)|C = y$. Since multiplication by a constant in a quasigroup is injective, and the operations are applied cellwise, we obtain that if $x_i \mapsto z \cdot x_i$ is injective, and the only possibility is that these patterns differ at $e_G$. Thus, $y$ has at least as many extensions to $e_G$ as $x$.
\end{proof}

Recall that group shifts are \emph{cellwiseable} \cite{Ki87,SaTo12e} meaning up to conjugacy the operations of the group can be taken cellwise (the term is from \cite{SaTo20}). Thus, this lemma indeed applies to all group shifts up to conjugacy. Quasigroup shifts are not cellwiseable \cite{SaTo12e}, so in the case of quasigroups the assumption that the alphabet itself is a quasigroup may be crucial.

\begin{example}
The Ledrappier subshift
\[ X_L = \{x \in \Z_2^{\Z^2} \colo x_{\vec v} + x_{\vec v + (1,0)} + x_{\vec v + (0,1)} = 0\} \]
is an abelian group shift with cellwise group operation addition in the two-element group $\Z_2$. In particular, it is a quasigroup subshift with cellwise operations, thus it has equal extension counts. Thus it is an avoshift for the family of all subsets of $\Z^2$. \qee
\end{example}

\begin{example}
Let $G$ be any f.g.\ infinite group. Then there is a group shift $X \subset \Z_2^G$ which is not uniformly $\pow(G \setminus \{e_G\})$-avo. For example, $X = \{0^G, 1^G\}$ is not uniformly $\pow(G \setminus \{e_G\})$-avo. \qee
\end{example}

Our next example are the $k$-TEP subshifts. In \cite{Sa22d} these are defined in high generality (with respect to any ``$S$-UCP translation-invariant convexoid $\C$'' on the shift group $G$), but we restrict here to the Euclidean case. Let $S \Subset \Z^d$, and let $\pat \subset \Sigma^S$ be a set of patterns. A \emph{corner} of $S$ is any $s \in S$ such that there exists a linear functional $\ell : \R^d \to \R$ such that $\ell(s) > \ell(t)$ for all $t \in S \setminus s$.

\begin{definition}
 We say $\pat$ is $k$-TEP if for all corners $s$ of $S$, for any $Q \in \Sigma^{S \setminus s}$ there are exactly $k$ patterns in $\pat$ which agree with $Q$ on $S \setminus s$. We say $X \subset \Sigma^{\Z^d}$ is $k$-TEP if it is defined by forbidden patterns $\Sigma^S \setminus \pat$ for some $k$-TEP set $\pat \subset \Sigma^S$. 
\end{definition}

Let $\D$ be the family of all intersections of real convex sets with $\Z^d$, and let $\C$ be the set of all $C \in \D$ such that $\vec 0 \notin C$ and $C \cup \{\vec 0\} \in \D$.

\begin{lemma}
Let $X \subset \Sigma^{\Z^d}$ be a $k$-TEP subshift. Then $X$ has uniform equal extension counts for $\C$, thus is uniformly $\C$-avo.
\end{lemma}

\begin{proof}
This is the restriction of Lemma~5.2 in \cite{Sa22} to the Euclidean case.
\end{proof}

\begin{definition}
Let $X \subset \Sigma^G$, and let $0 \in \Sigma$. We say $0$ is a \emph{safe symbol} for $X$ if whenever $y \in \Sigma^G$, and for some $x \in X$ we have $y_g \in \{x_g, 0\}$ for all $g \in G$, then $y \in X$.
\end{definition}

\begin{proposition}
Let $0 \in A$, and suppose $X \subset A^G$ is a subshift of finite type and $0$ is a safe symbol for $X$. Then $X$ is uniformly avo for $\pow(G \setminus \{e_G\})$.
\end{proposition}

\begin{proof}
Let $C \subset G \setminus \{e_G\}$ be arbitrary. Let $D \Subset C$ be large enough so that $d(e_G, C \setminus D)$ is larger than the diameter of any forbidden pattern. Suppose $x \in X$ and let $y \in \follow(x|D, e_G)$. Define a new configuration $z$ by
\[ z_g = \begin{cases}
y_g & \mbox{if } g \in D \cup \{e_G\} \\
x_g & \mbox{if } g \in C \\
0 & \mbox{otherwise}.
\end{cases} \]
Note that the first two cases intersect in $D$, but $x|D = y|D$ so the definitions agree.

We claim that $z$ does not contain any forbidden patterns. This is because in every $E$ which is the support of a forbidden pattern, $z|E$ agrees with either the configuration $x|C \sqcup 0^{G \setminus C}$ or $y|D \cup \{e_G\} \sqcup 0^{G \setminus (D \cup \{e_G\})}$, which are legal by the assumption that $0$ is a safe symbol.
\end{proof}







Next we consider the topological strong spatial mixing property \cite{Br18}, which generalizes the previous proposition.

\begin{definition}
\label{def:TSSM}
A subshift $X \subset A^G$ satisfies \emph{topological strong spatial mixing} or \emph{TSSM} with gap $n \in \N$, if for any disjoint sets $U, V, S \Subset G$ such that $d(U, V) \geq n$, and for every $u \in A^U, v \in A^V$ and $s \in A^S$, 
\[ u \sqcup s \sqsubset X, s \sqcup v \sqsubset X \implies u \sqcup s \sqcup v \sqsubset X \] 
\end{definition}

Brice\~no defines this in the case of $\Z^d$ in \cite{Br18}, but the definition makes sense in any group. He does not state that $U, V, S$ are disjoint, but this does not change the definition.

\begin{theorem}
\label{thm:TSSM}
A subshift has topological strong spatial mixing if and only if it is uniformly avo for $\pow(G \setminus \{e_G\})$.
\end{theorem}

\begin{proof}
Suppose a subshift has strong spatial mixing with gap $n$. Let $C \Subset G \setminus \{e_G\}$ be arbitrary, and let $B = C \cap \ball_n$. Suppose $x \in X|C$ and $y \in X|B \cup \{e\}$ with $x|B = y|B$. We should show $x \sqcup y_e^e \sqsubset X$. Pick large $m$, and take $U = (C \cap \ball_m) \setminus B$, $V = \{e_G\}$, $S = B$, $u = x|U$, $s = x|S$, $v = y_e^e$. Then the assumptions of TSSM are easy to check and we thus have $u \sqcup s \sqcup v = x|(U \cup S) \sqcup y_e^e \sqsubset X$. Since $m$ was arbitrary, we conclude that $x \sqcup y_e^e \sqsubset X$ by compactness.

Next suppose $X \subset \Sigma^G$ is uniformly $C$-avo with avoradius $n$ for all sets $C \subset G \setminus \{e_G\}$. Pick $n$ as the gap, and consider any disjoint sets $U, V, S \Subset G$ such that $d(U, V) \geq n$, and patterns $u \in A^U, v \in A^V$ and $s \in A^S$ such that $u \sqcup s \sqsubset X, s \sqcup v \sqsubset X$. Write $V = \{p_1, \ldots, p_m\}$.

In particular $u \sqcup s \sqsubset X, s \sqcup v_{p_1}^{p_1} \sqsubset X$, so the uniform avo property applied for the set $U \cup S$ at $p_1$, gives us $u \sqcup s \sqcup v_{p_1}^{p_1} \sqsubset X$ because $p_1 \ball_n \cap U \cup S \subset S$ meaning $s \sqcup v_{p_1}^{p_1} \sqsubset s \sqcup v \sqsubset X$ from the assumption.

We can proceed by induction, observing that if
\[ u \sqcup s \sqcup v_{p_1}^{p_1} \sqcup \cdots \sqcup v_{p_k}^{p_k} \sqsubset X, \]
then because
\[ s \sqcup v_{p_1}^{p_1} \sqcup \cdots \sqcup v_{p_k}^{p_k} \sqcup v_{p_{k+1}}^{p_{k+1}} \sqsubset X \]
and $p_{k+1} \ball_n \cap (U \cup S \cup V) \subset S \cup V$, we can apply the avo property for the set $U \cup S \cup \{p_1, \ldots, p_k\}$ at $p_{k+1}$ to deduce
\[ u \sqcup s \sqcup v_{p_1}^{p_1} \sqcup \cdots \sqcup v_{p_k}^{p_k} \sqcup v_{p_{k+1}}^{p_{k+1}} \sqsubset X. \]
We eventually obtain $u \sqcup s \sqcup v$ as desired.
\end{proof}

The following is related to Lemma~\ref{lem:ZCase} where one-sided intervals were used to characterize SFTs. Here $\C$ is the set of inductive intervals that are restricted to being negative in one axis (as we will see when we give the general definition). The proposition shows that in the two-dimensional case, this does not lead to $X$ being an avoshift for the two-sided version of $\C$.

Here, a \emph{cellular automaton} is a shift-commuting continuous function $f : \Sigma^{\Z^d} \to \Sigma^{\Z^d}$. A \emph{neighborhood} is $N \Subset \Z^{d}$ such that $x \mapsto f(x)_{\vec 0}$ is a function of $x|N$ (which always exists by continuity). The function $x|N \mapsto f(x)_{\vec 0}$ is called a \emph{local rule} for $f$. A cellular automaton is \emph{reversible} if it is bijective (this is equivalent to any of the following: injectivity, being a homeomorphism, or having a cellular automaton inverse $f^{-1}$). 

\begin{proposition}
\label{prop:Spacetime}
Let $f : \Sigma^\Z \to \Sigma^\Z$ be a surjective cellular automaton. Let $X$ be its spacetime subshift $\{x \in \Sigma^{\Z^2} \colo \forall i: x_{i+1} = f(x_i)\}$ where $x_i \in \Sigma^\Z$ is defined by $(x_i)_j = x_{j,i}$. Let $\C$ consist of sets $(I_1 \times \{0\}) \cup \Z \times I_2$ where $I_1$ is an interval of one of the forms $(-\infty,-1], [-n,-1], \emptyset, [1, n], [1, \infty)$ and $I_2$ is an interval of one of the forms $(-\infty, 1], [-n,-1]$. Then
\begin{itemize}
\item $X$ is always a $\C$-avoshift,
\item if $f$ is reversible, then $X$ is avo for the sets $\C$ and for the sets $-\C = \{-C \colo C \in \mathcal{C}\}$, and
\item if $\Sigma = \{0,1,2\}$ and $f$ is the cellular automaton with neighborhood $\{0,1\}$ and local rule given by the rules $F(2, a) = 2$ and $F(a, b) = a + b \text{ mod 2}$ otherwise (taken in $\{0,1\}$), then $X$ is not a $(-\C)$-avoshift.
\end{itemize}
\end{proposition}

\begin{proof}
For the first item, consider an inductive interval with axis intervals $(I_1, I_2)$. If $I_2 \neq \emptyset$, then the local rule of $f$ can be used to determine the unique symbol used at $\vec 0$ from the contents of $[-n, n] \times \{-1\} \subset \Z \times I_2$ if $f$ has neighborhood $[-n, n]$. If $I_2 = \emptyset$, any symbol is legal since $X|\Z \times \{0\}$ is just full shift on $\Sigma$ due to the surjectivity of $f$.

For the second item, if $f$ is reversible we can use the same argument for $f^{-1}$.

For the third item, consider the inductive interval with axis intervals $(I_1, I_2)$ where $I_2 = \{1\}$ and $I_1 = \emptyset$, i.e.\ the set $C = \Z \times \{1\}$. then we cannot determine from any finite subset of the all-zero pattern $0^C$ what the possible symbols at $(0,0)$ are (i.e.\ we cannot locally compute the set of possible symbols at the origina of $f$-preimages of $0^\Z$), since for $0^C$ either symbol can appear, but if we change a $0$ to $2$ far to the right, then the symbol $0$ becomes forced.
\end{proof}

It is a nice exercise (but possibly not an easy one) to show that in the previous proposition (even without assuming $f$ surjective), the spacetime subshift $X$ is an avoshift for $\C \cup -\C$ (the set of all inductive intervals) if and only if $f$ is stable (reaches its limit set in finitely many steps) and $f$ is a constant-to-one (equivalently, open) endomorphism of its limit set.

Next, we show that the avo property is also related to the topological Markov property. The reason we include this is that in \cite{BaGoMaTa20} the authors are able to prove this property for abelian group shifts on many groups, and some interesting properties of group shifts can be deduced by only using the topological Markov property. Specifically, this property suffices to show some interesting measure-theoretic properties. It is a much weaker property than being SFT; even on $\Z$, strong TMP subshifts need not be SFT, although they are sofic \cite{ChHaMaMePa14}, and on $\Z^2$, there are uncountably many subshifts with strong TMP, so they are far from even being sofic. Thus one cannot expect decidability results. 

\begin{definition}
Let $X \subset A^G$ be a subshift. We say $X$ has the \emph{topological Markov property} if for all $B \Subset G$ there exists $C \Subset G$ containing $B$ such that for all $x, y \in X$ with $x|C\setminus B = y|C \setminus B$, the point $z \in A^G$ defined by $z|C = x|C$ and $z|G \setminus B = y|G \setminus B$ is in $X$. If there exists $S \Subset G$ such that we can always take $C = BS$, then $X$ has the \emph{strong topological Markov property}.
\end{definition}

\begin{proposition}
\label{prop:TMP}
A subshift is (uniformly) avo for the family of cofinite subsets of $G \setminus \{e_G\}$ if and only if it has the (strong) TMP.
\end{proposition}

\begin{proof}
Suppose first that $X$ is avo for the cofinite subsets. Let $B \Subset G$ be arbitrary. Write $B = \{g_1, g_2, \ldots, g_n\}$ without repetition. For all $i$ the set $g_i^{-1}((G \setminus B) \cup \{g_1, \ldots, g_{i-1}\}) = F_i$ is a cofinite subset of $G \setminus \{e_G\}$. Thus, for any such $i$ there exists a finite set $D_i$ such that $D_i \Subset F_i$ and such that 
\[ \follow(x|D_i, e) = \follow(x|F_i, e) \]
for all $x \in X$. Writing $E_i = g_iD_i$, translating by $g_i$, and substituting $y = g_ix$, we deduce
\[ \follow(y|E_i, g_i) = \follow(y|(G \setminus B) \cup \{g_1, \ldots, g_{i-1}\}, g_i) \]
for all $y \in X$. Let $E = \bigcup_i E_i$. We claim that we can pick $C = B \cup E$ in the definition of TMP.

Namely, suppose $x, y \in X$ with $x|C\setminus B = y|C \setminus B$. We prove by induction (in finitely many steps) that the point $z \in A^G$ defined by $z|C = x|C$ and $z|G \setminus B = y|G \setminus B$ is in $X$, by showing that $z|(G \setminus B) \cup \{g_1, \ldots, g_i\}$ is globally valid. For $i = 0$, this is clear since $y \in X$. Now assume $z|(G \setminus B) \cup \{g_1, \ldots, g_{i-1}\}$ is globally valid. We have
\[ \follow(z|E_i, g_i) = \follow(z|(G \setminus B) \cup \{g_1, \ldots, g_{i-1}\}, g_i) \]
from the above, so $x_{g_i}$ is a legal symbol for extending $z|(G \setminus B) \cup \{g_1, \ldots, g_{i-1}\}$ to $g_i$ if and only if $z|E_i \cup \{g_i\}$ appears in $X$. But since $E_i \subset E \subset C$, we have $z|E_i \cup \{g_i\} = x|E_i \cup \{g_i\} \sqsubset X$.

Conversely, suppose we have the topological Markov property, and let $D \subset G \setminus \{e_G\}$ be any cofinite set, meaning $D = G \setminus B$ for some finite $B \ni e$. We will show that $X$ is $D$-avo. For this, let $C \Supset B$ be as in the definition of TMP for the set $B$. We claim that $C \setminus \{e_G\}$ is $D$-determining. Namely, suppose $y \in X|D$. We show the non-trivial direction
\[ y|(C \setminus \{e\}) \sqcup a^e \sqsubset X \implies y \sqcup a^e \sqsubset X. \]
For this, suppose $y|(C \setminus \{e\}) \sqcup a^e \sqsubset X$, say $x \in X$ satisfies $x|C = y|(C \setminus \{e\}) \sqcup a^e$.

Note that $x|C \setminus B = y|C \setminus B$, since $e \in B$. Thus by TMP, the unique point $z$ with $z|C = y|C, z|G \setminus B = y|G \setminus B$ is in $X$. But clearly $z|D \cup \{e\} = y \sqcup a^e$, so $y \sqcup a^e \sqsubset X$ as desired.

By retracing the proof one sees that the uniform avo property similarly corresponds to strong TMP.
\end{proof}


\section{Being SFT in an area}

In this section, we introduce the idea of a set of configurations in a subset of a group being SFT. Let $X \subset A^G$ be a subshift, and let $S \subset G$.

We say $X$ is \emph{SFT on $S$} if there is a finite set of finite patterns $\pat$ such that $x \in X|S$ if and only if $x \in A^S$ and none of the patterns in $\pat$ appear as a subpattern in $x$. When $X \subset A^S$ and $S \subset G$ with $G$ clear from context, we also say directly that $X$ is \emph{SFT} if there is a finite set of forbidden patterns with domains $D \Subset G$, such that $x \in X$ if and only if $x \in A^S$, and $x$ does not contain a translate of one of these patterns whose domain fits inside $S$ (even if $X$ might not extend to any subshift on $G$). If $\C \subset \pow(G)$, then we say $X$ is \emph{uniformly SFT on the sets $\C$} or simply \emph{uniformly $\C$-SFT} if there is a finite set of finite patterns $\pat$ that define all the restrictions $X|C$ for $C \in \C$. In each case a \emph{window} is a set $M \Subset G$ which can contain the domains of all the defining forbidden patterns, and a \emph{window size} is $m$ such that $\ball_m$ is a window.

We make some simple general observations about this notion. These results are not used in the following sections. Nevertheless, the first two propositions may be useful for understanding the notion, and the reader may find the third observation of independent interest.

\begin{proposition}
\label{prop:AllOnFinite}
Every subshift is SFT on any (single) finite area.
\end{proposition}

\begin{proof}
Let $X \subset A^G$ be a subshift, and let $S \Subset G$. Let $\pat$ be all $S$-pattern that do not appear in $X$. Then trivially $\pat$ defines $X|S$.
\end{proof}

\begin{proposition}
A subshift is SFT if and only if it is SFT on every (single) cofinite area.
\end{proposition}

\begin{proof}
Since $G$ is cofinite in $G$, a subshift SFT on all cofinite sets must itself be an SFT.

For the other direction, let $X \subset A^G$ be an SFT, and let $B \Subset G$. Let $S \subset G$ be a symmetric set containing $e_G$ such that a defining set of forbidden patterns $\pat \subset A^S$ exists. Let $\qat$ be the set of all patterns with domains contained in $BS^2$. If $x \in X$, then clearly $x|G \setminus B$ does not contain an occurrence of any pattern from $\qat$. Conversely, suppose $x \in A^{G \setminus B}$ does not contain any pattern from $\qat$. Then in particular $P = x|BS^2 \setminus B$ appears in $X$, thus there exists a legal pattern $Q \in X|BS^2$ with $Q|BS^2 \setminus B = P$. Define $y \in A^G$ by $y|G \setminus B = x$ and $y|B = P$, so that in fact $y|BS^2 = Q$. If $g \in BS$, then $gS$ is contained in $BS^2$ and thus $y|gS^2 \sqsubset Q \sqsubset X$. If $g \notin BS$, then $gS \subset G \setminus B$, so $y|gS \sqsubset x \in X$. Thus, $y$ does not contain any of the defining forbidden patterns, and we conclude $y \in X$.
\end{proof}


\begin{proposition}
A subshift is uniformly SFT in finite sets of every fixed cardinality if and only if it is topologically $k$-mixing for all $k$.
\end{proposition}

To clarify the reading, let $\C_n$ be the set of sets of cardinality $n$. The left-hand-side of the equivalence states that for every $n$, $X$ is uniformly SFT on the sets $\C_n$ (but not necessarily uniformly in $n$).

\begin{proof}
Suppose first that $X \subset A^G$ is uniformly SFT in all finite sets of a fixed cardinality $n$. Let $P_1, \ldots, P_k$ be any finite set of patterns that appear in $X$. Let $n$ be the sum of cardinalities of their supports. Let $\pat$ be a set of forbidden patterns defining the restrictions $X|C$ for all $C \in \C_n$. Let $r_1$ be the maximal norm of any element of the domain of any pattern in $\pat$, and let $r_2$ be the maximal norm of the domain of any pattern from any $P_i$, then the pattern $g_1P_1 \cup \cdots \cup g_kP_k$ cannot contain any pattern from $\pat$ as long as $d(g_i, g_j) > 2r_1 + 2r_2$ (since any such pattern can see at most one translate of a pattern $P_i$, and all of them appear in $X$). Since $k$ was arbitrary, we conclude that $X$ is topologically $k$-mixing for all $k$.

For the converse, we show by induction on $k$ that if $X$ is topologically $k$-mixing for all $k$, then it is uniformly SFT in $\C_{k, r}$, the sets whose support is contained in a (not necessarily disjoint) union of $r$-balls. For $k = 1$, this follows from Proposiiton~\ref{prop:AllOnFinite}, since $\C_{k, r}$ contains only finite sets for any $r$. Suppose now the claim holds up to $k$, and all $r \in \N$. We prove the claim or $k+1$, for an arbitrary $r \in \N$.

By $(k+1)$-mixing, there exists $n$ such that for any $k+1$ many patterns with supports contained in the $r$-ball, and which are globally legal, any union of those patterns whose separation (distance between the centers of the containing $r$-balls) is at least $n$ is also globally legal. Thus, we can find a finite set of forbidden patterns such that the set of patterns with support $C \in \C_{k+1,r}$ are correctly defined, when $C$ can be partitioned into $k+1$ many $r$-balls whose separation is at least $n$. We need to show that we can add enough forbidden patterns so that also the patterns whose domain $C \in \C_{k+1,r}$ cannot be partitioned this way are correctly defined.

But note that for any such $C$, we can break its domain into equivalence classes, by putting two of the $r$-balls in the same class when their distance is less than $n$. Each single equivalence class fits into a $R = ((k + 1)n + 2r)$-ball. Thus, it suffices to apply induction to topological $\ell$-mixing and this choice of $R$.
\end{proof}

\subsection{Connections between the (uniform) avo and uniform SFT properties}
\label{sec:UniformAvo}

In this section, We show that
\begin{itemize}
\item under minor assumptions on $\C$, uniform SFT implies uniform avo,
\item under very strong assumptions on $\C$, avoshift implies uniformly avo, and
\item under intermediate assumptions on $\C$, uniform avo implies uniform SFT. 
\end{itemize}

\begin{definition}
We say a family $\C \subset \pow(G \setminus \{e_G\})$ is \emph{good} if
\[ C \in \C \implies \exists g \in G: g(C \cup \{e_G\}) \in \C. \]
\end{definition}

\begin{lemma}
\label{lem:SFTimpliesAvo}
If $X \subset \Sigma^G$ is uniformly SFT on the sets $\C \subset \pow(G \setminus \{e_G\})$, and suppose $\C$ is good. Then $X$ is uniformly avo on $\C$.
\end{lemma}

\begin{proof}
Let $\pat$ be a set of finite patterns that defines $X$ on the sets $\C$. Let $x \in \Sigma^C$ be globally valid for $X$. We need to show
\[ x \sqcup a^e \sqcup X \iff x|B \sqcup a^e \sqcup X \]
for some finite $B \Subset C$. Pick $B = \ball_r \cap C$ such that $\ball_r$ contains the domain of every $P \in \pat$. Now if $x|B \sqcup a^e \sqcup X$, then $x|B \sqcup a^e$ does not contain any pattern from $\pat$. Since $x \in X|C$, neither does $x$. Thus, neither does $x \sqcup a^e$. Since $\C$ is good, we have $D = g(C \cup \{e\}) \in \C$ for some $g \in G$, thus $\pat$ defines the restriction $X|D$ as an SFT, thus it defines also $X|g^{-1}D$ since $X$ is shift-invariant. This is the domain of $x \sqcup a^e$, so the locally valid configuration $x \sqcup a^e$ is globally valid in $X$, proving the uniform avo property.
\end{proof}

Recall that a \emph{well-quasi-order} or \emph{wqo} is a preorder $<$ such that every infinite sequence contains a nondecreasing subsequence. We only consider partial orders $<$ (in fact, set containment). Equivalently, a wqo is a preorder that is well-founded (there are no infinite decreasing sequences) and has no infinite antichains (infinite sets of incomparable elements).

\begin{lemma}
\label{lem:CommonRadius}
Let $\C$ be any family of subsets of $G \setminus \{e_G\}$ which is wqo under inclusion, and closed under increasing union. Suppose $X$ is $\C$-avo. Then $X$ is uniformly $\C$-avo.
\end{lemma}

\begin{proof}
It suffices to show that a single avoradius works for all $C \in \C$. Suppose not, and for each $i \in \N$, let $C_i \in \C$ be any set that does not admit radius $i$. By the wqo assumption, we may restrict to a subsequence $C_{n_i}$ which are increasing as sets, and do not admit radius $n_i$. Then $\bigcup_i C_{n_i} = C \in \C$ by closure under increasing union, and thus $C$ admits a determining set $B \Subset C$ by the $\C$-avo assumption. For all large enough $i$, we have $B \Subset C_{n_i} \subset C$. Thus, $B$ is also a determining set for the intermediate set $C_{n_i}$. Taking $i$ also large enough so that $B \subset \ball_{n_i} \cap C_{n_i}$, we conclude that $n_i$ is an avoradius for $C_{n_i}$, a contradiction.
\end{proof}

If $S$ is a set where an order is understood from context, then ${\downarrow s} = \{t \in S \colo t < s\}$. Note that typically with this notation one has $s \in {\downarrow s}$, but not here.

\begin{definition}
\label{def:Constructible}
Let $\C$ be a family of subsets of $G \setminus \{e_G\}$. A \emph{$\C$-well-ordering} $<$ of $D \subset G$ is a well-ordering of $D$ such that for each $s \in S$, the set $s^{-1} {\downarrow s} = \{s^{-1}t \colo t < s\}$ is in $\C$. If $D$ has a $\C$-well-ordering then we say $D$ is \emph{$\C$-constructible}. A family $\C$ is \emph{constructible} if every $C \in \C$ admits 
a $\C$-well-ordering, and the group $G$ also admits one. 
\end{definition}

\begin{lemma}
\label{lem:UniformSFT}
Let $X \subset A^G$ be a subshift that is uniformly $\C$-avo for $\C \subset \pow(G \setminus \{e_G\})$ with avoradius $r$. Then whenever $D$ is $\C$-constructible, $X|D$ is SFT with window size $r$. 
\end{lemma}

\begin{proof}
If $X$ is empty, the claim is trivial as we may forbid the empty pattern to define its restriction to any set $D$.

Consider now any constructible set $D$. Let $\pat$ be the set of all patterns that do not appear in configurations of $X$, and whose domain is contained in the $r$-ball. Of course if $x \in X|D$, then $x$ does not contain any of these patterns. Suppose then that $x \in A^D$ does not contain any of these patterns. We will show that $x$ is the restriction of a point in $X$. 

Take a $\C$-well-ordering $<$ of $D$, and recall the notation ${\downarrow s} = \{t \in S \colo t < s\}$. Note that $D$ is isomorphic to an initial segment of the ordinals, and also isomorphic to $\{{\downarrow s} \colo s \in D\}$ (ordered by containment) under $s \mapsto {\downarrow s}$, with limit ordinals corresponding to increasing unions. We show by induction along $<$ that $x|{\downarrow s}$ is globally valid in $X$. This is obviously true for the minimal element $s_0$ of $D$, since $x|{\downarrow s_0}$ is the empty pattern, and $X$ is nonempty. For limit ordinals, this follows immediately from compactness of $X$ (an increasing union of globally valid patterns is globally valid).

For successor ordinals $\downarrow s$, say with predecessor $t \prec s$ for $s$, we are dealing with a $<$-prefix $C \sqcup \{t\} \subset D$ with maximal element $t$, such that $x|C$ is globally valid and $t^{-1}C \in \C$. By shift-invariance of $X$, also $y|t^{-1}C$ is globally valid where $y = t^{-1}x$. Since $r$ is a radius for $t^{-1}C$, the set $B = \ball_r \cap t^{-1}C$ is determining for $t^{-1}C$.

Since $y|t^{-1}C \in X|t^{-1}C$ (this is just another way to say that $y|t^{-1}D$ is globally valid), the definition of a determining set states that
\[ {y|B \sqcup a^e \sqsubset X} \iff {y|t^{-1}C \sqcup a^e \sqsubset X}. \]
We have indeed $y|B \sqcup a^e \sqsubset X$, since $B \cup \{e\}$ is contained in $\ball_r$, and $y|B \sqcup a^e = y|B \cup \{e\}$ is a pattern that appears in $X$ (since $y = t^{-1}x$ does not contain any pattern from $\pat$). Thus, $y|t^{-1}C \cup \{e\} = y|t^{-1}C \sqcup a^e \sqsubset X$.

Shifting back by $t$, we conclude that $x|C \cup \{t\} = x|C \sqcup a^t$ is globally valid, concluding the induction step.
\end{proof}


\section{Polycyclic groups and inductive intervals}

Let $G$ be polycyclic. Fix a sequence of subgroups $1 = H_0 < H_1 < \cdots < H_n = G$ each normal in the next, so that quotients $H_{i+1}/H_i$ are cyclic (the existence of such a sequence is the definition of a polycyclic group). Pick $h_i$ so that $h_i H_{i-1}$ generates $H_i/H_{i-1}$, equivalently $H_i = \langle h_i \rangle H_{i-1}$. Let $k_i = \mathrm{ord}(H_i/H_{i-1})$.

\begin{definition}
The $((h_i)_i, (H_i)_i, (k_i)_i)$ are called a \emph{polycycle structure} for $G$.
\end{definition}

Of course $H_i$ and $k_i$ are determined by the choices of $h_i$, but it is convenient to always have $h_i, H_i, k_i$ refer to this data. Each $H_i$ by default inherits a polycycle structure from $G$, by restricting to an initial subset of the $h_i$. We always consider a polycyclic group as carrying a fixed polycycle structure. We call $n$ the \emph{size} of the structure. The various $i$ are informally referred to as \emph{axes} or \emph{dimensions}, and a dimension is called \emph{finite} or \emph{infinite} depending on whether $k_i < \infty$.

Note that the size $n$ is \emph{not} equal to the standard \emph{Hirsch length} of the group, which only counts infinite dimensions. A polycyclic group is called \emph{strongly polycyclic} if all the $k_i$ are infinite. Every polycyclic group is virtually strongly polycyclic, and by a bit of recoding it should be possible to remove the finite dimensions from the following discussion completely, but the theory should be more easily applicable if they are allowed, and they do not add much length to the discussion. 

Each element of $G$ corresponds uniquely to a tuple $t = (t_1, \ldots, t_n) \in \Z^n$ where for all $i$ such that $k_i < \infty$ we have $t_i \in [0, k_i-1]$. The tuple is given inductively for $g \in G$: The last coordinate is a direct projection to $H_n/H_{n-1}$, then we turn it to $0$ by multiplying by a power of $h_i$ from the right, and extract the tuple for $H_{n-1}$ inductively (using its inherited polycycle structure) to get the first $n-1$ coordinates. Conversely we write the element $g$ corresponding to a tuple $t$ as $g(t)$. We are typically only interested in the value of the last nonzero element of the tuple, which is independent of the choice of the $h_i$ (it only depends on the $H_i$).

We say an interval $I \subset \Z$ (possibly infinite in one directions) \emph{grazes $0$} if it is of one of the forms $(-\infty,-1], [-n, -1], \emptyset, [1, n], [1,\infty)$. It is \emph{positive} if it is one of the last two, and \emph{negative} if it is one of the first two. This is the \emph{sign} of the interval. The empty set has no sign \emph{zero}, say. In the case of a finite cyclic group $\Z_k$, intervals are images of intervals on $\Z$ under projection. Then an interval graces $0$ if it is empty, or it contains one of $1$ or $-1$ (or both), but not $0$.

\begin{definition}
Let $G$ be a polycyclic group with a polycycle structure as above. Its \emph{inductive intervals}, or \emph{II}s, are defined inductively as sets $S \subset G$ of elements whose tuple $t$ satisfies that either
\begin{itemize}
\item $t_n$ comes from a particular interval gracing $0$ in $\Z_{k_n}$, or
\item $t_n = 0$ and $(t_1, \ldots, t_{n-1})$ comes from a certain inductive interval in $H_{n-1}$.
\end{itemize}
\end{definition}

Concretely, an II $I$ is thus determined by an $n$-tuple of intervals $(I_1, I_2, \ldots, I_n)$ where $I_i$ is a $0$-gracing interval in $\Z_{k_i}$ (of $k_i$ finite) or $\Z$ (if $k_i = \infty$). The $I_i$ are called \emph{axis intervals}. Then $g(t_1, \ldots, t_n) \in I$ if $t_i \neq 0$ for some $i$, and for the maximal such $i$ we have $t_i \in I_i$.

For the group $\Z^d$, we by default use the polycyclic structure where $h_i$ is the $i$th standard generator $(0, \ldots, 0, 1, 0, \ldots, 0)$ with $1$ in the $i$th position. Again, Figure~\ref{fig:minecraft} illustrates an inductive interval for the group $\Z^3$ with this choice. For example, since the last axis interval is $[-64, -1]$ and this axis points upward, the ground below the player is completely included in the set, all the way until $-64$; but on the first axis (pointing right), we only fill finitely many positions.



\subsection{Main properties of inductive intervals}

\begin{lemma}
Inductive intervals are a good family.
\end{lemma}

\begin{proof}
If an inductive interval $C$ is given by axis intervals $(I_1, I_2, \ldots, I_n)$ where $I_1 = [1, m]$ or $I_1 = [1, \infty)$ then $C \cup \{e_G\}$ is given by $(J, I_2, \ldots, I_n)$ where $J = [0, m]$ or $J = [0, \infty)$, and then $g_1C$ is an inductive interval, where $g_1$ is the generator for the first axis. The same is true if $I = \empty$. If  $I_1 = [-m, -1]$ or $I_1 = (-\infty, -1]$, then $g_1^{-1}C$ is an inductive interval.
\end{proof}

We show that the family of inductive intervals has good order-theoretic properties, and that polycyclic groups can be ordered so that all prefixes of the order are also inductive intervals up to translation. These are the properties that later allow us to conclude that II-avo subshifts are SFTs, as well as computable properties.

\begin{lemma}
\label{lem:wqo}
Inductive intervals form a well-quasi-order under inclusion.
\end{lemma}

\begin{proof}
Consider a sequence of inductive intervals $(I_1^i, \ldots, I_n^i)_i$. If some inductive interval appears infinitely many times, we are done. If some $I_j^i$ takes on only finitely many intervals, we may restrict to the infinite subsequence where this interval is fixed. Next, we may restrict to an infinite subsequence such that each $I_i$ is of the same sign, i.e.\ grazes $0$ from the same side. By symmetry, we may assume that $I_i$ is empty or $1 \in I_i$ for all $i$. We may assume that, for all $j$, if $I_j^i$ does not stay fixed, it becomes a longer and longer prefix of $[1, \infty)$. The sequence of IIs we end up with is then increasing under set containment.

For the latter claim, observe that $s$-signed inductive intervals are a subfamily of the inductive intervals.
\end{proof}

\begin{lemma}
\label{lem:Union}
Inductive intervals are closed under increasing union. 
\end{lemma}

\begin{proof}
Let $(I_1^i, \ldots, I_n^i)_i$ be an increasing sequence of inductive intervals. This means just that each sequence of intervals $I_j^i$ is increasing in $i$. The union is easily verified to be $(\bigcup_i I_1^i, \ldots, \bigcup_i I_n^i)$. Again the same is true automatically for $s$-signed inductive intervals.
\end{proof}

In particular it follows from the above lemmas that the ($s$-signed) inductive intervals are also topologically closed in $\pow(G)$.

\begin{lemma}
\label{lem:Closed}
The inductive intervals are topologically closed in $\pow(G)$.
\end{lemma}

\begin{proof}
If we have a converging sequence of inductive intervals, the previous lemmas show we can find an increasing subsequence that converges to an inductive interval. Thus the sequence itself converges to an inductive interval.
\end{proof}

\begin{lemma}
\label{lem:IIWellOrder}
The family $\C$ of inductive intervals is constructible. 
\end{lemma}

\begin{proof}
We start with the claim that inductive intervals are constructible. We now prove the statement of the lemma by induction on the dimension $n$. For $n = 1$, the claim is easy, simply enumerate prefixes of the unique axis interval $I_1$ in the case of an inductive interval, and for the entire group $\Z$, enumerate first the nonnegative numbers, and then the negative numbers (for example).

For general $n$, first consider an inductive interval $(I_1, \ldots, I_n)$. Up to symmetry we may assume that $1 \in I_i$ for all nonempty $I_i$, i.e.\ the intervals are not negative. Let $\ell_n \leq k_n$ be the maximal element of the last axis interval $I_n$ (or $\ell_n = \infty$ if $k_n = \infty$ and the interval is $[1, \infty)$). Pick a $\C$-well-ordering of the group $H_{n-1}$ by induction. Shift this well-ordering to $H_{n-1} \times \{h_n^j\}$ for $j < \ell_n$ by picking any base point. Then order $H_{n-1} \times \{h_n^j \colo 0 < j < \ell_n\}$ by first comparing the power of $h_n$, and then $H_{n-1}$ in case of equality. This is a $\C$-well-ordering, since for any individual element the translated downset $S = s^{-1} {\downarrow s}$ is of the form $S = (H_{n-1} \times \{h_n^{-1}, \ldots, h_n^{-j}\}) \cup T$ where $T \subset H_{n-1}$ is a prefix of the order. Since $T$ is an inductive interval on $H_{n-1}$, $S$ is one in $G$. 

At this point we have $\C$-well-ordered the subset with axis intervals $(\emptyset, \ldots, \emptyset, I_n)$. Next, we well-order $(I_1, \ldots, I_{n-1})$ by induction, as an inductive interval on $H_n$). We add this at the end of the well-order described in the previous paragraph. The order remains a $\C$-well-order, since the new translated downsets are those of $(I_1, \ldots, I_{n-1})$, with all of $I_n$ included on the last axis.

In the case of the entire group $G$, mix the previous idea and the case of $\Z$: first list the positive cosets, and then the negative ones.
\end{proof}

\begin{definition}
\label{def:IIExtendable}
If $D \subset G$ is such that $G \setminus D$ admits a well-order such that $s^{-1} ({\downarrow s} \cup D)$ is in $\C$ for all $s \notin s$, then we say $D$ is \emph{$\C$-extendable}. We say a family of shapes $\C$ is \emph{extendable} if every $C \in \C$ is $\C$-extendable. 
\end{definition}

\begin{lemma}
\label{lem:IIWellOrder2}
The inductive intervals are extendable.
\end{lemma}


\begin{proof}
Consider an inductive interval $D$ with axis intervals $(I_1, \ldots, I_n)$. Begin ordering the complement by adding elements on the sides of $I_1$ in, say, alternating, order. It is easy to see that up to a shift, the prefixes of the order, together with $D$, are of the form $(J, I_2, \ldots, I_n)$ for larger and larger intervals $J$ with alternating signs, and after adding all elements of the first axis this way, we have ordered a translate of $(\emptyset, I_2', \ldots, I_n)$ where now $I_2'$ is $I_2$ with one new element. We can now order a new coset of the first axis (on either side of $I_2'$). We can continue similarly up the dimensions to order $D$ (with order type at most $\omega^n$).
\end{proof}



\section{SFTness and the avo property in concrete examples}

\subsection{Polycyclic groups}

\begin{theorem}
\label{thm:PolycyclicUniformSFT}
Let $G$ be a polycyclic group, and let $X \subset A^G$ be a subshift that is avo for the class $\C$ of all inductive intervals. Then $X$ is uniformly SFT on $\{G\} \cup \C$. 
\end{theorem}

\begin{proof}
Suppose $G$ is a polycyclic and $X \subset A^G$ is avo for inductive intervals $\C$. By Lemma~\ref{lem:wqo}, $\C$ is a wqo under inclusion, and by Lemma~\ref{lem:Union} it is closed under increasing union. 
Thus by Lemma~\ref{lem:CommonRadius}, $X$ is uniformly $\C$-avo. By Lemma~\ref{lem:IIWellOrder}, $\C$ is constructible, meaning any $D \in \C \cup \{G\}$ admits a $\C$-well-order. Then Lemma~\ref{lem:UniformSFT} implies that $X|D$ is a subshift of finite type with window size $r$ for any of these sets. In other words, $X$ is uniformly SFT on $\{G\} \cup \C$.
\end{proof}

\begin{corollary}
Let $G$ be a polycyclic group, and let $X \subset A^G$ be a subshift that is avo for inductive intervals. Then $X$ is SFT.
\end{corollary}

\begin{proof}
$G \in \{G\} \cup \C$.
\end{proof}

For group shifts, the following is shown in \cite{Sc95}. We obtain a new proof.

\begin{corollary}
Let $G$ be a polycyclic group, and let $X \subset \Sigma^G$ be a quasigroup shift (i.e.\ $\Sigma$ is a quasigroup and $X$ is closed under cellwise quasigroup operations). Then $X$ is SFT.
\end{corollary}

\begin{proof}
We showed in Lemma~\ref{lem:Quasigroup} that $X$ is avo for all subsets of $G \setminus \{e\}$. In particular it is avo for iterated intervals, thus it is SFT.
\end{proof}

For shift groups beyond polycyclic, one cannot expect to obtain SFTness from an avo property. Namely, every group that has a non-f.g.\ subgroup admits a group shift which is not of finite type \cite{Sa18e}. Group shifts are avo for all sets, so in particular there cannot exist any family of shapes on such a group, which would allow deducing SFTness of avoshifts for that shape. (Finitely-generated groups where all proper subgroups are finitely-generated are known to exist, see e.g.\ \cite{Ol83}, but simple examples are not known.)

For the lamplighter group $\Z_2 \wr \Z$, we note that there is even a sofic group shift which is not SFT, so ``avo implies SFT'' is not even true for sofic shifts, for any family of sets $\C$. (The same can be proved on many other groups with a sufficiently strong simulation theory.)

\begin{proposition}
On the lamplighter group $G$, there is a sofic group shift (thus an avoshift for all subsets $G \setminus \{e\}$), which is not SFT.
\end{proposition}

\begin{proof}
Every group shift is avo for the family of all sets, so it suffices to find a sofic group shift which is not SFT. The two-symbol group shift on $\bigoplus_{i \in \Z_+} \Z_2$ is of course not SFT \cite{Sa18e}. It is easy to see that also its pullback in the sense of \cite{BaSa24} to $\bigoplus_{i \in \Z} \Z_2$ is not SFT but is still a group shift. By \cite{Sa18e}, its free extension to any supergroup (i.e.\ the subshift defined on the larger group by the same forbidden patterns) is also non-SFT but is still a group shift, in particular we get a non-SFT group shift on the lamplighter group. By the simulation theorem of \cite{BaSa24}, this specific form of subshift is automatically sofic, since the two-symbol subshift on $\bigoplus_{i \in \Z_+} \Z_2$ is obviously computable. 
\end{proof}

\subsection{Free group SFTs are avo}

We recall a definition from \cite{Sa22d}.

\begin{definition}
Let $G$ be free on generators $S$. A \emph{tree convex set} of $G$ is $C \Subset G$ such that if $u, w \in C$ and $v$ lies on the geodesic between $u, w$, then if $\min(d(v, u), d(v, w)) = r$, the ball $v \ball_{r-1}$ is contained in $C$. A \emph{limit tree convex set} is a possibly infinite set with the same property.
\end{definition}

\begin{lemma}
A set is a limit tree convex set if and only if it is a limit of tree convex sets in Cantor space.
\end{lemma}

\begin{proof}
The condition is clearly closed so limits of tree convex sets are limit tree convex. Let then $C$ be limit tree convex. Then $C_r = C \cap \ball_r$ satisfies the condition, so is tree convex for all $r$, and $C$ is a limit of these sets.
\end{proof}

We also need the standard notion of a \emph{geodesically convex set}: $C \subset G$ is geodesically convex in a free group $G$ if the unique path between any $a, b \in C$ is contained in $C$.

If $\C$ is a family of sets, the corresponding \emph{extension sets} are $\D = \{C \subset \pow(G \setminus \{e\}) \;|\; C \in \C \wedge C \cup \{e\} \in \C\}$.

\begin{theorem}
\label{thm:FreeGroupSFTUniformlyAvo}
A subshift on a free group is an SFT if and only if it is uniformly avo on the extension sets of limit tree convex sets.
\end{theorem}

\begin{proof}
Let $G$ be a free group. Suppose first that $X$ is an SFT. Consider $x \in X|C$ where $C$ and $\dot C = C \cup \{e_G\}$ are limit tree convex sets. Let $X$ be an SFT with forbidden patterns of diameter at most $r$. Recall that $\dot C$ is geodesically convex \cite{Sa22d}. If there is no geodesic path from $e_G$ which is of length at least $2r$, we have a bound on the diameter of $\dot C$ and we are done.

Since the radius-$r$ ball is a cutset of the group, when determining the possible legal continuations we need not look past any such ball. For any geodesic path of length at least $2r+1$ in $C$, the $r$-ball around the central element is contained in $C$ since $C$ is a limit tree convex set. We conclude that it suffices to look a bounded distance away from the identity element to know the legal extensions, which implies the uniform avo property.

Suppose then that $X$ is uniformly avo on the limit tree convex sets. In particular it is uniformly avo with some avoradius $r$ on the tree convex sets. By \cite{Sa22d}, the tree convex sets form a convexoid. This implies that we can order $G$ with order type $\omega$ so that every prefix of it is tree convex. Then Lemma~\ref{lem:UniformSFT} implies that $X$ is a subshift of finite type with window size $r$.
\end{proof}

\begin{example}
\label{ex:FreeGroupGeodesics}
On the free group $F_2 = \langle a, b \rangle$, there is an SFT which is not avo for geodesically convex sets. Consider an SFT over alphabet $\Z_2^2$ with the rule that if $x_{gb} = (s_1, t_1), x_{gab} = (s_2, t_2), x_{g} = (s_3, t_3)$ then $t_3 = s_1 + s_2$. Then $b^{-1} \langle a \rangle \cup \{b^{-1} a^n b\}$ is geodesically convex and has avoradius $\Omega(n)$. \qee
\end{example}

\subsection{More applications of Lemma~\ref{lem:UniformSFT}}

Lemma~\ref{lem:UniformSFT} (which deduces uniform SFT from uniform avo) has implications on groups other than polycyclic ones. It only needs a constructible family of sets, and these are easy to find. We illustrate the result by proving the following result from \cite{Br18}, which applies to all groups.

\begin{proposition}
Let $X \subset A^G$ be a topologically strong spatial mixing subshift on any f.g.\ group $G$. Then $X$ is SFT.
\end{proposition}

\begin{proof}
Recall that TSSM is equivalent to being uniformly avo for the class of all sets $\C \subset \pow(G \setminus \{e_G\}$. Obviously $\C$ is constructible (one can well-order any set in any way). By Lemma~\ref{lem:UniformSFT} $X|D$ is SFT with window size $r$ for any $D \subset G$, in particular this is true for $D = G$.
\end{proof}

Less trivial examples exist of constructible families, for example the convexoids $\C$ (slightly generalizing the standard \emph{convex geometries}) considered in \cite{Sa22d} have this property. A \emph{convexoid} is a family $\C$ of finite subsets of a set $G$ which satisfies $\emptyset \in C$ and $\forall B \Subset G: \exists C\in\C: C \supset B$, and which have the \emph{corner addition property}
\[ C, D \in \C \wedge C \subsetneq D \implies \exists a \in D \setminus C: C \cup \{a\} \in \C \]

\begin{lemma}
\label{lem:ConvexoidComputable}
Let $\C$ be a convexoid. Then the family of extension sets of $\C$ is constructible. 
\end{lemma}

\begin{proof}
Let $\D$ be the family of extension sets, i.e.\ all $C \in \C$ such that $e_G \notin C$ and $C \cup \{e_G\} \in \C$. We check that each $D \in \D$ and $G$ itself are $\D$-constructible, i.e.\ admit $\D$-well-orders.

First, we recall that repeated application of the corner addition property for $C \Subset D$ provides an \emph{anti-shelling from $C$ to $D$}, or ordering of $D \setminus C$ such that any prefix of this order, together with $C$, is convex (meaning belongs to the convexoid $\C$). If $D \in \C$, then one can verify that an anti-shelling from $\emptyset$ to $D$ is a $\D$-well-order.

From the second property in the definition, we can list an increasing sequence of convex sets $C_1, C_2, \ldots$. Starting from $C_0 = \emptyset \in \C$, and concatenating anti-shellings from $C_i$ to $C_{i+1}$ for all $i$, one can check that the resulting order of $G$ (with order type $\omega$) is a $\D$-well-order.
\end{proof}

Thus for example the extension sets of tree convex sets defined in Section~\ref{sec:Examples} are extendable. We could thus conclude from Lemma~\ref{lem:UniformSFT} that the uniform avo property on the free group, for extension sets of tree convex sets, implies the SFT property. Since we already characterized SFTs are the uniformly avo subshifts for the extension sets of limit tree convex sets, this is not worth stating here (in the uniform avo case, taking the limits does not matter). Convexoids for strongly polycyclic groups are built in \cite{Sa22d}, as well as on some other groups which are not discussed in the present paper.




\section{Decidability results}
\label{sec:Decidability}

Next, we explain how one can, starting from a finite set of forbidden patterns, prove (algorithmically or in any sufficient proof system) that an SFT $X$ is indeed an avoshift for a set of shapes and compute its language and perform comparisons between such subshifts. In the case of inductive intervals on polycyclic groups, we can also compute traces (restrictions to subgroups in the subnormal series). In the next section we explain how one can also compute some factors (avofactors), which in the case of group shifts include all algebraic factors. 

A common situation in symbolic dynamics is that we are applying partial algorithms to undecidable problems, and we know that they work on some subclass of the problems, but this class itself is not obviously decidable, we usually want that our algorithm never make an incorrect claim (in addition to making correct claims on the subclass). Namely, this allows us to blindly apply our algorithms without even knowing whether they belong to the class. The algorithm is ``allowed'' to work correctly even work on some instances that are not in this class, or it may never halt, it is simply not allowed to give an incorrect answer. We make this slightly more precise in the following definition.

\begin{definition}
Let $I$ be some countable set of (codings of) problem instances, $J \subset I$ a subset of these instances (thought of as a nice subclass), and $T \subset I$ the ``positive instances''. We say an algorithm \emph{safely partially solves $(I, T)$} if, when given $i \in I$, it either eventually outputs ``yes'', ``no'', or never halts. Furthermore, if it says ``yes'' then $i \in T$, and if it says ``no'' then $i \notin T$. We say the algorithm \emph{safely solves $(I, T)$ on $J \subset I$} if it safely partially solves $(I, T)$ and, given $i \in I$, then it eventually gives an answer (thus indeed the correct answer) if $i \in J$.
\end{definition}

This definition readily generalizes from decision problems to problems with more complex output (like computing the language of a subshift).



We define some computability properties for the sets $\C$. A family of subsets of $G$ is \emph{upper semicomputable} if we can uniformly in $r$ compute a sequence of upper approximations to $B = \{C \cap \ball_r \colo C \in \C\}$, which eventually reaches $B$ (though we cannot necessarily detect when this has happened). We say such a family $\C$ is \emph{computable} if it is also lower semicomputable, meaning we can actually compute the set of finite subsets of sets in $\C$ (note that this does not necessarily mean we can decide $C \in \C$ for a given finite set $C$).

\begin{lemma}
Let $G$ be a finitely-generated group and let $\C \subset \pow(G)$ be any topologically closed family of shapes. Let $X$ be an SFT defined by a finite set of forbidden patterns $\qat$, such that $\qat$ uniformly defines $X$ on all $C \in \C$. Then for all $r \in \N$ there exists $R$ such that whenever $C \in \C$ and $x \in \Sigma^{C \cap \ball_R}$ is locally $\qat$-valid in $X$, then $x|C \cap B_r$ is globally valid in $X$.
\end{lemma}

\begin{proof}
Suppose that the conclusion fails for some fixed $r$. Then we can find for arbitrarily large $R$ a shape $C_R \in \C$, and a locally valid pattern $x_R \in \Sigma^{C_R \cap \ball_R}$ (containing no $\qat$-pattern), such that $x|C_R \cap B_r$ is not globally valid. 

Since $\C$ is closed, we can pass to a subsequence of $C_R$ that tends to $C \in \C$, in which case also $C_R \cap \ball_R$ tends to $C$. By compactness of Cantor space, we can find a converging subsequence of the patterns $x_R$ that tends to a configuration $x$ on $C \in \C$ which contains no pattern from $\qat$

By the assumption that $\qat$ defines the restriction $X|C$, we have $x \in X|C$. But $x|\ball_r \cap C = x_R|C_R \cap \ball_r \not\sqsubset X$ for large enough $R$, a contradiction.
\end{proof}


\begin{lemma}
\label{lem:ComputingPatterns}
Let $G$ be a finitely-generated group with decidable word problem, and let $\C \ni \emptyset$ be a topologically closed, upper semicomputable, extendable family of shapes. Then there is an algorithm that, given a finite set of patterns $\pat$ defining an SFT $X$ on $G$, enumerates all finite families of forbidden patterns $\qat$ that uniformly define the $\C$-restrictions of $X$.
\end{lemma}

In particular, this algorithm recognizes the SFTs that are uniformly $\C$-SFTs. Note that here we need not explicitly discuss safety; 
 the algorithm simply will not enumerate any sets $\qat$ if such sets do not exist.

\begin{proof}
Since $X$ is SFT and $G$ has decidable word problem, we can compute upper approximations to the language. Note also that the equality of SFTs is semidecidable (see e.g.\ \cite{SaTo23}), so we can enumerate the complete list of sets of patterns $\qat$ which define $X$.

Thus it suffices to show that once we have found a set of patterns $\qat$ equivalent to the original, which additionally defines $X|C$ uniformly on $C \in \C$, then we can recognize this fact. 
For this, consider any such $\qat$ and let $R$ be given for $2r+1$ obtained from the previous lemma for the family $\qat$, meaning if $x \in \Sigma^{C \cap \ball_R}$ is a locally $\qat$-legal configuration, then $x|C \cap \ball_{2r+1}$ is globally legal.

We claim that then whenever $x \in \Sigma^{C \cap \ball_R}$ is $\qat$-legal for $C \in \C$, then $x$ has a $\qat$-legal extension to $e_G$. 
Namely, otherwise every extension $x|(C \cap \ball_{2r+1}) \sqcup a^e$ is actually globally illegal. But this is a contradiction, since clearly $x|(C \cap \ball_{2r+1})$ then already cannot be a globally valid pattern. 


This means that we can (for the fixed $\qat$) go through larger and larger $R$, and for fixed $(\qat, R)$ we can compute better and better approximations to the set of $R$-sized prefixes of shapes in $\C$. Since the limits of shapes from $\C$ are upper semicomputable, we eventually know the correct set of their subshapes of size at most $R$ (though of course we cannot necessarily recognize that we do). At this point, by the assumption on $R$ and the discussion of the previous paragraph, no matter what $\qat$-legal pattern we put on one of these shapes $M \Subset \ball_R \cap G \setminus \{e_G\}$, it will have at least one legal continuation to $e_G$ that does not contain a $\qat$-pattern.

At this point, the algorithm can conclude that $\qat$ is a set of forbidden patterns defining $X|C$ uniformly on $C \in \C$. Namely, if we have a configuration on $C \in \C$ which has no $\qat$-pattern, then by $\C$-extendability of $C$, we can order $G \setminus C$ so that prefixes of the order together with $C$ are (up to a shift) in $\C$, and inductively prove along this order that at any point of the process there is an extension containing no pattern from $\qat$.
\end{proof}

\begin{lemma}
Let $G$ be a finitely-generated group with decidable word problem, and let $\C \ni \emptyset$ be a topologically closed, upper semicomputable and extendable family of shapes. Then there is an algorithm that, given a finite set of patterns $\pat$ defining an SFT $X$ on $G$, safely solves emptiness on all $X$ which are uniformly SFT on $\C$. If $\C$ is constructible, it also safely solves emptiness on uniform $\C$-avoshifts $X$.
\end{lemma}

\begin{proof}
If the given SFT $X$ is empty, we can recognize this fact. Otherwise, we will try to apply the previous theorem. Since its output is correct on all SFTs (when it gives an output), if it eventually outputs some $\qat$, we can determine emptiness: The empty pattern is in the language of a subshift if and only if the subshift is nonempty. Since $\emptyset \in \C$, the only possible reason why the empty pattern would be forbidden is that the empty pattern is directly in $\qat$. I.e.\ $X$ is empty if and only if $\qat \cap \Sigma^\emptyset \neq \emptyset$. If $X$ is uniformly SFT on $\C$, then eventually we find $\qat$ can can check this.

It suffices to show the last claim. Indeed, if $\C$ is constructible and $X$ is a uniform $\C$-avoshift, then $X$ is uniformly SFT on $\C$ by Lemma~\ref{lem:UniformSFT}.
\end{proof}

The following theorem implies in particular the theorem of \cite{BeKa24} that projective subdynamics of group shifts on $\Z^d$ can be computed effectively.

\begin{theorem}
\label{thm:PolycyclicQat}
Let $G$ be a polycyclic group with polycycle structure $((h_i)_i, (H_i)_i, (k_i)_i)$. Then given an SFT $X$, we can safely compute a set of forbidden patterns $\pat$ that defines $X$ uniformly on all inductive intervals if $X$ is an II-avoshift. The restriction $X|H_i$ is an avoshift for all $i$. Furthermore, we can effectively compute the forbidden patterns for $X|H_i$.
\end{theorem}

\begin{proof}
For the first claim, by Theorem~\ref{thm:PolycyclicUniformSFT}, $X$ is uniformly SFT on iterated intervals whenever $X$ is an II-avoshift. By Lemma~\ref{lem:ComputingPatterns}, we can find a set of forbidden patterns $\qat$ defining it uniformly on these sets, since by Lemma~\ref{lem:IIWellOrder2} the II are extendable (the other properties are obvious).

The second claim (that $X|H_i$ is an avoshift) is immediate from the definition.

For the third (decidability) claim, since each of the groups $H_i$ is an iterated interval, the forbidden patterns in $\qat$ that admit a translate contained in $H_i$ directly define $X|H_i$. Obviously it is decidable whether a forbidden pattern admits such a translate.
\end{proof}

\begin{corollary}
The set of II-avoshifts on a polycyclic group is recursively enumerable.
\end{corollary}

\begin{proof}
For each finite set of forbidden patterns defining an SFT $X$, we try to compute forbidden patterns that define $X$ uniformly on all inductive intervals. If we find such, then $X$ is indeed uniformly SFT on the inductive intervals. Since the inductive intervals are a good family, Lemma~\ref{lem:SFTimpliesAvo} implies that $X$ is (uniformly) avo on inductive intervals.
\end{proof}

Next, we proceed to show that we can compute the language of an avoshift under suitable assumptions.

\begin{lemma}
\label{lem:DecidableLanguage}
Let $G$ be a group with decidable word problem, let $\C$ be a topologically closed, computable family of sets, and suppose that $G$ is $\C$-constructible. Then there is an algorithm that, given a finite set of forbidden patterns $\pat$ such that $\pat$ defines an SFT $X$ $\C$-uniformly, decides the language of the SFT $X$.
\end{lemma}

This is a promise problem, i.e.\ we don't care what the algorithm does if $\pat$ does not have the property stated. (But we will only apply this when we actually know $\pat$ has this property.)

\begin{proof}
Our algorithm computes better and better upper approximations to the set of patterns of $X$ on finite subsets of $G$ using upper semicomputability of the language.

At times, it will declare that it has found the precise set of patterns in the language on a particular domain. We will be sure to only add such deduction rules when they are actually safe (assuming the given property of $\pat$), i.e.\ when the algorithm declares it knows the patterns of a particular shape, this can be trusted to be actually true.

To root the process, the algorithm can check whether the empty pattern is in $\pat$. If it is, then $X$ is empty and we are done. Otherwise, it concludes that it knows the exact set of globally legal patterns on the domain $\emptyset$ already (i.e.\ the empty pattern appears in the language). Also, if the algorithm has already deduced the globally legal patterns on a domain $D$, then it knows the globally legal patterns on any $gD$, and on any subset of $D$.

The crucial point is that if we know the exact patterns on $C \cap \ball_R \setminus \{e_G\}$ for $C \in \C$, where $R$ is larger than the diameter of any domain of a pattern in $S$, then we also know the $(C \cap \ann_R) \cup \{e_G\}$-patterns. This is because a pattern on $\dot C = C \cup \{e_G\}$ is in $X$ if and only if it does not contain a $\pat$-pattern, so if we know the contents of a $C$-pattern on the annulus $\ball_R \setminus \{e_G\}$, then since patterns in $\pat$ cannot see over the annulus, the algorithm just needs to list the possible extensions that do not directly introduce a pattern from $\pat$, and these will give exactly the globally legal patterns on the extended shape.

We can now prove by transfinite induction on $s \in G$, w.r.t.\ the $\C$-well-order of $G$, that the algorithm eventually knows all $D$-shaped patterns for finite sets $D$ whose elements are strictly below $s$, and all their translates. (Stated like this, if the order on $G$ has a maximum $s$, then we do not actually deduce the finite patterns containing $s$ directly, but any pattern can be shifted so as not to include the maximum.) For limit ordinals $s$, the induction step is trivial. For successor ordinals, suppose $t \prec s$ is the predecessor and let $D \Subset C = \{g \colo g < t\}$ be any finite set. It suffices to show that the algorithm eventually deduces the correct patterns on $D \cup \{t\}$.

By induction, we eventually know the exact set of patterns on $\ball_R \cap t^{-1} C$. If $R$ is larger than the diameter of any pattern in $\pat$ and large enough so that $t\ball_r \supset D$, then we note that by the main deduction rule, the algorithm eventually deduces the set of patterns on $(\ball_R \cap t^{-1} C) \cup \{e_G\}$. After a shift, it deduces them on $(C \cap t\ball_R) \cup \{t\}$, and then also the subset $D \cup \{t\}$.
\end{proof}

\begin{lemma}
\label{lem:InclusionSafelyDecidable}
Let $G$ be a finitely-generated group with decidable word problem. Let $A$ be any algorithm that, given an SFT $X$, safely enumerates the language of every $X$ if $X$ belongs to a class of subshifts $\mathcal{S}$. Then given SFTs $X, Y \subset \Sigma^G$, the inclusion $X \subset Y$ is safely decidable for pairs $(X, Y)$ such that $X \in \mathcal{S}$.
\end{lemma}

\begin{proof}
The inclusion $X \subset Y$ is semidecidable. If $X \not\subset Y$, then there exists a pattern $P \sqsubset X$ such that $P \not\sqsubset Y$. The latter is semidecidable. The former is semidecidable if $X \in S$, because eventually the algorithm $A$ outputs the pattern $P$. Since $A$ outputs the language of $X$ safely, if it outputs $P$ then indeed $P \sqsubset X$, so if the deduction $X \not\subset Y$ is made, it is correct even if we do not necessarily have $X \in S$, so $X \subset Y$ is indeed decided safely.
\end{proof}


While the previous lemma is not commonly stated as we have here, it is well-known and commonly applied in the case that $\mathcal{S}$ is the class of subshifts where periodic points are dense. Namely, Wang's algorithm \cite{Wa61,SaTo23} is more or less the proof of precisely this result. Note that here the idea of ``safety'' is quite useful -- Wang's algorithm naturally allows use to conclude $X \not\subset Y$ for also many $X$ where periodic points are not dense, since the algorithm that enumerates the subpatterns of periodic points indeed safely enumerates the language of a given SFT when it has dense periodic points.

\begin{theorem}
The language of an II-avoshift on a polycyclic group is safely computable, uniformly in the SFT. In particular, given two SFTs on a polycyclic group $X, Y$, such that $X$ is an II-avoshift, we can safely decide the inclusion $X \subset Y$.
\end{theorem}

\begin{proof}
By Theorem~\ref{thm:PolycyclicQat} we can safely find a set of forbidden patterns defining $X$ uniformly on the IIs. Then we can safely compute the language by Lemma~\ref{lem:DecidableLanguage}. The previous lemma concludes the proof.
\end{proof}

We recall that even in the case $G = \Z$, the previous theorem covers some cases Wang's algorithm does not, as it does not require dense periodic points. (Similarly, Wang's algorithm covers many cases ours does not.)

The following is an easy exercise in automata theory, but it may be of some interest that the avoshift technology covers this.

\begin{theorem}
The language of an SFT on the free group is decidable. In particular, given two SFTs on a free group $X, Y$, we can decide the inclusion $X \subset Y$.
\end{theorem}

\begin{proof}
As we showed in Section~\ref{sec:Examples}, any SFT on the free group is uniformly avo on the extension set $\D$ of the limit tree convex sets $\C$. Lemma~\ref{lem:UniformSFT} says that if $X \subset A^G$ is a subshift that is uniformly $\D$-avo for $\D \subset \pow(G \setminus \{e_G\})$ with avoradius $r$, then whenever $D$ is $\D$-constructible, $X|D$ is SFT with window size $r$. The set $\D$ is constructible since $\D$ is a convexoid by Lemma~\ref{lem:ConvexoidConstructible}. 

Next, Lemma~\ref{lem:ComputingPatterns} says that if $\D \ni \emptyset$ is a topologically closed, upper semicomputable, extendable family of shapes, then there is an algorithm that, given the finite set of patterns $\pat$ defining an SFT $X$ on $G$, enumerates all finite families of forbidden patterns $\qat$ that uniformly define the $\C$-restrictions of $X$. Topological closure and upper semicomputability are clear for $\D$ from the corresponding properties of $\C$ (which are clear from the definition). Extendability follows from the proof of constructibility in Lemma~\ref{lem:ConvexoidConstructible} since we can take $C_i = C$ for some $i$ in the proof. We conclude that such a set $\qat$ exists, and thus the algorithm eventually finds one.

Next, Lemma~\ref{lem:DecidableLanguage} says that if $\D$ be a topologically closed, computable family of sets, and $G$ is $\D$-constructible, then there is an algorithm that, from such a set $\qat$, decides the language of the SFT $X$. Since the family $\qat$ is guaranteed to indeed define $X$ uniformly on $\D$, the language is decided safely for all SFTs.

Lemma~\ref{lem:InclusionSafelyDecidable} then implies that inclusion of SFTs is safely decidable for all SFTs, in other words the problem is simply decidable.
\end{proof}

\section{Avofactors}
\label{sec:Avofactors}

We slightly generalize the avoshift concept, to give an avoshift proof for the useful fact that algebraic factors of group shifts can be computed \cite{BeKa24}. We give only the outline, as this is analogous to the theory developed above.

Let $S$ be any set, and $\C \subset \pow(S) \times S$ any set of \emph{cornered subsets} meaning $(C, c) \in \C \implies c \notin C$. Then we say $X \subset \Sigma^S$ is $\C$-avo if for any $(C, c) \in \C$, there is a finite subset of $B \Subset C$ such that $\follow_X(x|C, c) = \follow_X(x|B, c)$ for all $x \in X$. In the case of subshifts, we could always take $c = e_G$ and always translate the cell to be determined to the origin, so we did not discuss corners explicitly.

Now, if we have a family of subsets $\C$ on a group $G$, then we can extend it to a family of cornered shapes $\D$ on $G \times L$ where $L = \{1, \ldots, \ell\}$, by taking the shapes $(G \times \{j+1, \ldots, \ell\}) \cup (C \times \{j\})$ with corner $(e_G, j)$. This corresponds roughly to the inductive intervals on the group $G \times \Z_{\ell}$, which on the last axis graze the origin from the right, but now $\{1, \ldots, \ell\}$ is not a group but an interval. Now we define an \emph{avoshift relation} to be a subset $X \subset \Sigma^{G \times \{1, \ldots, \ell\}}$ which is closed, $G$-invariant, and is avo for the cornered set described in this paragraph. The construction order one should imagine is that we build the ``cosets'' $G \times \{i\}$ for decreasing $i$, and the possible constructions in each coset come from $\C$.

In the case that $\C$ is the set of inductive intervals on a polycyclic group, one can repeat the arguments of the previous sections to conclude that if $X \subset \Sigma^{G \times L}$ is $G$-invariant and a $\D$-avoshift, then we can compute a set of $G$-invariant finite patterns on $G \times L$ that define $X$ uniformly on all sets in $\D$.

\begin{definition}
Let $G$ be a polycyclic group with fixed polycycle structure. Given an avoshift $X \subset \Sigma^G$, we say $Y \subset \Sigma^G$ is an avofactor of $X$ if there is a factor map $f : X \to Y$ such that the graph $Z \subset X \times Y \subset \Sigma^{G \times \{1, 2\}}$ of $f$, defined by $(x, y) \in Z \iff f(x) = y$ is an avorelation.
\end{definition}

Note that we assume the same alphabet just for notational convenience; if they are not the same, we can consider both with the union alphabet. It is important that the image is written on the second component, so that the ``image of $f$ is constructed before the preimage'' when we consider $\D$-well-orderings. The decidability result below effectively comes from the fact that we find local rules for building preimages for any image.

Now we can show the following analogously to Theorem~\ref{thm:PolycyclicQat}.

\begin{theorem}
Every avofactor of an II-avoshift on a polycyclic group is an avoshift, and we can effectively compute its forbidden patterns.
\end{theorem}

\begin{proof}[Sketch of proof]
Let $\C$ be the inductive intervals of the polycyclic group $G$. Consider the subshift $Z \subset X \times Y \subset \Sigma^{G \times \{1, 2\}}$, assumed avo. Note that $Y$ must be an avoshift, since $C \times \{2\} \in \D$ for all $C \in \C$, proving the first claim.

Next, we go through the entire theory with $\D$ in place of inductive intervals. We can easily show that $\D$ is wqo under inclusion and closed under increasing union. Now the proof of Lemma~\ref{lem:CommonRadius} shows that any $\D$-avoshift $X$ is uniformly $\D$-avo.

Next, analogously to Lemma~\ref{lem:UniformSFT} we see that $X|D$ is uniformly SFT for $D \in \D$, i.e.\ there is a set of forbidden patterns for it, by showing that sets in $\D$ are constructible in an appropriate sense.

Since $\D$ are also extendable and are computable in any reasonable sense, we see as in Section~\ref{sec:Decidability} that we can compute a set of forbidden patterns defining $X$ uniformly on $D \in \D$. In particular these patterns work on $G \times \{2\} \in \D$, which means precisely that they give forbidden patterns for the image subshift.
\end{proof}
 
\begin{corollary}
We can effectively compute the algebraic factors of any group shift on a polycyclic group.
\end{corollary}

\begin{proof}
It is easy to verify that the graph of a shift-invariant continuous group homomorphism between group shifts is an avorelation.
\end{proof}

\begin{example}
\label{ex:FullShift}
Not all factors of avoshifts are avo -- in fact, avoshifts on inductive intervals are not closed under conjugacy even on the group $\Z^2$ (note that on $\Z$ they are just the SFTs, so they are closed under conjugacy). For example, take the $\Z^2$ binary full shift. Now double the alphabet and always write the $f$-image of the first track of the row below, to the second track of the present row, where $f$ is the cellular automaton from Example~\ref{ex:FreeGroupGeodesics}. As in Example~\ref{ex:FreeGroupGeodesics}, it is easy to see that the subshift is not avo for inductive intervals. \qee
\end{example}


\section{Discussion and future work}
\label{sec:Discussion}

The avoshift notion we have developed here already has strong implications, and already generalizes many of the known important properties of group shifts. However, the theory is still in its infancy, and we outline here some future directions.

\subsection{Practical implementations}
\label{sec:Practical}

One motivation of the present paper is to provide a (first approximation to a) theoretical framework for practical decision procedures for basic operations on SFTs, at least comparing them and computing their languages. We have not attempted to make the procedures efficient in the present paper (and do not even provide pseudocode), nor have we implemented them. Our hope is that they can be implemented with SAT solvers. If they end up being practical, the plan is also to fit them in the framework of \cite{SaTo23}.

For now we do not know how practical the methods are, as we have not applied them (on the computer). A particularly interesting question is how well these methods perform compared to Wang-style algorithms for deciding properties of multi-dimensional SFTs. From our experiments with the Diddy toolbox \cite{SaTo23} (which implements Wang's algorithm), our impression is that the basic algorithm of Wang from \cite{Wa61} for computing the language of a subshift using only density of periodic points is remarkably effective also for comparing SFTs when implemented with SAT solvers.

While of course multidimensional SFTs in practice tend to have periodic points and counterexamples are always carefully engineered rather than arise naturally, \emph{dense periodic points} is a stringent property that fails for many natural examples, in particular it is trivial to find pairs of SFTs whose comparison cannot be made with Wang's algorithm directly. It is indeed trivial to find examples where the avoshift theory can check noninclusion of SFTs, but Wang's algorithm cannot, even in one dimension.

This is not a true weakness of Wang style algorithms: one could say that the core idea of Wang's algorithm is that one can prove the global validity of a finite pattern, by giving a complete finitary description of a configuration, in a sufficiently simple logical framework so that one can check it belongs to the subhsift, and that finding such a description is reasonable to find. Periodic points are of course an obvious example of such structures, but one may more generally consider semilinear points (see e.g.\ \cite{SaTo22c}) possible with restrictions on the eventual periods (which Diddy has very limited support for, at present \cite{SaTo23}), and it is plausible that one can more generally use automatic structures \cite{KhMi07}. These already cover a lot of ground, and in fact we are not aware of any avoshifts that do not have dense semilinear points. Thus, we are not aware of any deep theoretical reason why Wang-style methods could not be better than avoshifts.

On the other hand, there are situations where there is an easy-to-describe way to extend any configuration on a subset, but a simple configuration can nevertheless have an incredibly complicated extension with even uncomputable properties (think of spacetime diagrams of a Turing complete reversible cellular automaton). This might intuitively suggest that there are subshifts where an avoshift-style theory can make nontrivial deductions, but truly Wang-type algorithms are a logical impossibility in the sense that global configurations do not have direct finite descriptions (at least on the level of concreteness of automatic configurations). Again, we do not have concrete examples of such a phenomenon.

There are of course also other decidability methods that the present methods should be compared to. For example, in one dimension, classical automata theory seems unrivaled in its expressivity and practicality. One can also use automata theory in two dimensions by applying classical automata theory to slices of configurations. The references \cite{JeRa15,SaTo22c} are able to prove some interesting results about concrete examples of $\Z^2$-SFTs this way.

\subsection{Extensions to be explored}

In the case of $\Z$, in Lemma~\ref{lem:ZCase} we only used one-sided inductive intervals on $\Z$ (only the infinite one, because for finite sets the avo propery is automatic). One should indeed be able to extend the present theory so that it would cover polycyclic groups and ``one-sided/signed inductive intervals''. However, constructibility fails for these sets, and one has to consider weak constructibility, meaning we do not directly well-order sets, but present them as a well-ordered union of constructible sets. We believe the theory can be easily augmented to cover this, and also in two dimensions is covers some additional examples (such as spacetime diagrams of cellular automata Proposition~\ref{prop:Spacetime}). The main reason for not attempting to write the results in this generality is that while we can extend the SFTness proofs to this case, we do not know how to extend any of the decidability theory without the property of extendability (Definition~\ref{def:Extendable}), which fails for any notion of one-sided/signed inductive intervals. These decidability aspects are what we consider the most important facet of the theory developed here.

The theory of avoshifts could be extended to subshifts on monoids like $\Z \times \N$ or $\N^2$. However, we do not know what the optimal framework is. One might consider something like ``signed polycyclic groups'' where some of the infinite dimensions are restricted to only including the positive cosets of the normal subgroup. However, requiring intermediate dimensions to be positive can lead to complications in the non-abelian case (even in simple examples like the Heisenberg group), and we do not know what the cleanest approach is.

One possible solution is to develop the theory more generally in a setting of graph subshifts, to hopefully both cover groups like $\Z \times \N$, and to make the decidability theory of avofactors in Section~\ref{sec:Avofactors} a direct corollary of the general theory instead of needing an ad hoc discussion.

\subsection{Dynamical properties}

Avoshifts as defined her lack some desirable properties. First, they are not conjugacy invariant, thus being an avoshift is itself not actually a dynamical property (except on $\Z$ where it corresponds to being an SFT). This is the case also for SFTs with safe symbols and topological strong spatial mixing subshifts, as well as $k$-TEP subshifts, which as we recall as some of classes generalized by avoshifts. However, group shifts (arguably the main class of motivating examples) can be defined purely abstractly as internal groups in the category of subshifts. While we have to some extent swept this under a rug, they are indeed avoshifts only up to conjugacy. The failure is drastic here, as full shifts admit a group shift structure, but there are subshifts conjugate to full shifts which are not avo (Example~\ref{ex:FullShift}).

The mixing properties finite extension property \cite{BrMcPa18}, as well as the related recently defined map extension property \cite{Me23} and contractibility \cite{PoSa24}, all imply dense periodic points on a large class of groups \cite{PoSa24}, and are conjugacy invariant. The finite extension property was actually introduced as a natural conjugacy-invariant generalization of the topological strong spatial mixing property. Of course, one can trivially extend the present theory to cover the class of subshifts that are avoshifts up to conjugacy, in the sense that SFTness is conjugacy-invariant, and for the decidability theory one can go through possible conjugacies and the avo property through a conjugating map can be recognized. However, it seems likely that there is a more natural approach.

Finally, in the present paper we have only concentrated on decidability results. Group shifts, $k$-TEP subshifts and TSSM subshifts also have many other desirable properties such as dense periodic points, and interesting things can be said about the invariant measures in at least the first two cases (for TSSM and SFTs with a safe symbol, see \cite{Pa14} for some information). We do not know how much one can say, in the present generality, about the dynamical properties of avoshifts (other than the fact they are SFT). 

As we have shown, avoshifts do not have dense periodic points in general. For decidability theory, as discussed in Section~\ref{sec:Practical}, this is in a sense a good thing, as we want the class to cover as much dynamical ground as possible, and dense periodic points is a restrictive property. In a future paper, our plan is to find subclasses of avoshifts with more restricted dynamics (and still covering group shifts). Specifically, for avoshifts on convex sets in $\Z^d$, and for subshifts with equal extension counts, we claim that there are non-trivial things that can be said. 

\bibliographystyle{plain}
\bibliography{../../../bib/bib}{}

\end{document}